\documentclass[11pt,reqno,oneside]{amsart}

\usepackage{graphicx}
\usepackage{stmaryrd}
\usepackage{enumerate}
\usepackage[all]{xy}
\usepackage{amssymb,amsmath}
\usepackage{amsthm,amsfonts}
\usepackage{mathrsfs}
\usepackage{amstext}
\usepackage{enumerate}
\usepackage[normalem]{ulem}
\usepackage{color}
\newcommand{\hquotient}{\mathbin{\!\sslash\!}}

\newcommand{\cal}{\mathcal}
\newcommand{\back}{\overleftarrow}
\newcommand{\forth}{\overrightarrow}
\newcommand{\into}{\hookrightarrow}
\newcommand{\mpath}{\cal P_{\!\!\!\!{}_M}}
\newcommand{\mloop}{\Omega_{\!{}_M}\!}

\def\cl{\mathfrak c}
\def\op{\mathfrak o}
\def\ac{{\sf Act}}
\def\SC{\mathcal{SC}}
\def\acts{\mathrel{\reflectbox{$\curvearrowleft$}}}
\def\racts{\mathrel{\reflectbox{$\curvearrowright$}}}
\def\acinfty{\makebox{  $\acts  \hskip -13.2pt \raisebox{1.5pt}{${}_{{}_{\infty}}$} \,$ }}
\def\racinfty{\makebox{ $\racts \hskip -11.5pt \raisebox{1.5pt}{${}_{{}_{\infty}}$} \,$ }}

\newtheorem{thm}{Theorem}[section]
\newtheorem*{thm*}{Theorem}
\newtheorem{defi}[thm]{Definition}
\newtheorem{obs}[thm]{Observation}
\newtheorem{lem}[thm]{Lemma}
\newtheorem{prop}[thm]{Proposition}	
\newtheorem*{prop*}{Proposition}	
\newtheorem{coro}[thm]{Corollary}

\newtheorem*{nota}{Notation}

\begin{document}

\title{$A_\infty$-actions and Recognition of Relative Loop Spaces}
\author{Eduardo Hoefel}
\address{Universidade Federal do Paran\'a, Departamento de Matem\'atica C.P. 019081, 81531-990 \mbox{Curitiba}, PR - Brazil }
\email{hoefel@ufpr.br}
\author{Muriel Livernet}
\address{Universit\'e Paris 13, CNRS, UMR 7539 LAGA, 99 avenue
  Jean-Baptiste Cl\'ement, 93430 
  \mbox{Villetaneuse}, France}
\email{livernet@math.univ-paris13.fr}
\author{Jim Stasheff}
\address{Emeritus at the Mathematics Department of the University of North
Carolina at \mbox{Chapel} Hill – UNC-CH, Chapel Hill, NC 27599-3250}
\email{jds@math.upenn.edu}
\date{\today}

\begin{abstract}
We show that relative loop spaces are recognized by $A_\infty$-actions. A certain version of the 2-sided bar construction is used to prove such recognition theorem. The operad $\ac_\infty$ of $A_\infty$-actions is presented in terms of the Boardman-Vogt resolution of the operad $\ac$. We exhibit an operad homotopy equivalence between such resolution and the $1$-dimensional Swiss-cheese operad $\mathcal{SC}_1$.
\end{abstract}

\maketitle


\section{Introduction}
Given an $A_\infty$-space $X$ and a topological space $P$, an $A_\infty$-action of $X$ on $P$ is, by definition, an $A_\infty$-map $X \to {\rm End}(P)$. In this paper, we show that any $A_\infty$-action is weakly equivalent to the $A_\infty$-action of a loop space on a relative loop space. It is also shown that the weak equivalence is $A_\infty$-equivariant. 

The history of the notion of homotopy actions goes somewhat like the following.
Starting in 1895 with Poincar\'e's definition of the fundamental group of a topological space \cite{poincare}, the concept of homotopy associativity became well established, but the higher homotopies had to wait some six decades to be recognized \cite{sugawara:h}. Meanwhile, the corresponding notion of a homotopy action (or action up to homotopy) was at least implicit in the action of the fundamental group $\pi_1(X)$ on the set $\pi_1(X,A)$
for $A\subset X$. The next step was taken by Hilton\footnote{
The third author learned it as a graduate student from Hilton's {\it Introduction to Homotopy Theory} \cite{hilton:book}, the earliest textbook on the topic, published in 1953.}  \cite{hilton:book} in considering the long exact sequence associated to a fibration $F \to E \to B$,
$$\cdots\to \pi_n(F)\to \pi_n(E)\to \pi_n(B)\to \pi_{n-1}(F)\to \cdots$$
ending with $$\cdots \to \pi_1(B)\to \pi_0(F)\to \pi_0(E)\to \pi_0(B).$$
Of course, exactness is very weak at the end since the last three are in general only sets, but
exactness at $ \pi_0(F)$ is in terms of the action of $ \pi_1(B)$ on $ \pi_0(F).$  This passage to homotopy classes obscures the `action' of $\Omega B$ on $F$.  Initially, this was referred to as a \emph{homotopy action}, meaning only that $\lambda(\mu f)$ was homotopic to $(\lambda\mu)f$  for $f\in F$ and $\lambda, \mu \in \Omega B.$

If $X$ is a topological monoid, Nowlan \cite{nowlan}
developed the notion of an $A_\infty$-action of $X$ fibrewise on a fibration $p:P\to B$ and hence of an 
$A_\infty$-principal fibration when $X$ is the fiber of $p$. Nowlan's results were generalized to the case in which 
$X$ is itself an $A_\infty$-space by Iwase and Mimura \cite{iwase-mimura}. Given an $A_\infty$-space $X$ and a map $h : X \to P$, they introduced the notion of $A_\infty$-action along $h$. 

In the present paper, the map $h$ as well as its higher analogous maps $h_n : X^{\times n} \to P$ are part of an $\ac_\infty$-algebra structure on the pair $(X, P)$, where the operad $\ac_\infty$ is defined in section \ref{sec:w(act)}. We also show that the structure of an $\ac_\infty$-algebra on the pair $(X, P)$ is naturally equivalent to the existence of a $A_\infty$-map $X \to {\rm End}(P)$.

The homotopy orbit space $P \hquotient X$, defined originally by Iwase and Mimura, will be presented in terms of the 2-sided bar construction and will play a crucial role in the present paper. Our main result, Theorem \ref{thm:mainresult}, says that 
any $A_\infty$-action is $A_\infty$-equivariantly weakly equivalent to a relative loop space action, thus providing a natural rectification result for $A_\infty$-actions. The concepts of $A_\infty$-action and $A_\infty$-equivariance are treated in detail in section \ref{sec:recog}. Rectification of homotopy algebras in the context of colored operads has been studied in more generality by Berger and Moerdijk \cite{BerMoe07}. 

The operad $\ac_\infty$ is defined in section \ref{shaction} in terms of the Boardman-Vogt resolution of $\ac$. In section \ref{sec:swiss_cheese} we discuss in the details the relations between $A_\infty$-actions and the 1-dimensional Swiss-cheese operad $\SC_1$.  The $n$-dimensional Swiss-cheese operad $\SC_n$ will be used, in a sequel to this paper, to prove corresponding results for relative $n$-loop spaces with $n \geqslant 1$.

\subsection{Conventions and Notation} In this paper, all spaces are assumed to be compactly generated, Hausdorff and to have non-degenerate base points. All products and spaces of maps are taken with respect to the compactly generated topology. Throughout this paper, a \emph{weak equivalence} will always mean a weak homotopy equivalence between topological spaces. 
All the operads considered in this paper are non-symmetric.  Concerning colored operads, we will follow the definitions and notation of \cite{BerMoe07}. Most of the colored operads involved are $2$-colored, i.e., the set of colors $\{ c_1, c_2 \}$ has two elements. An algebra over a $2$-colored operad consists of a pair of spaces $(E_{c_1},E_{c_2})$. Such pairs are not to be confused with the topological pairs $(Y,B)$ where $B$ is a subspace of $Y$. Whenever it is clear from the context that 
the color $c_1$ comes before the color $c_2$, given any pair $(X,P)$ one can assume that $X$ is the space of color $c_1$ while $P$ has color $c_2$. For simplicity we will use the notation: 
\begin{equation*}
 x_{i...j} = (x_i, \dots, x_j) \in X^{\times (j - i + 1)}
\end{equation*}  for any set $X$ and any integers $i < j$.

\subsection*{Acknowledgments} The first author is supported by CNPq (grant 237798/2012-3) for a long term visit to Universit\'e Paris 13 and is grateful to its hospitality. This work is partially supported by CAPES/COFECUB Research Project 763/13 ``Factorization Algebras in Mathematical Physics and Algebraic Topology'' coordinated by the second author. The authors are grateful to the Isaac Newton Institute and to the organizers of the ``Grothendieck-Teichm\"uller Groups, Deformation and Operads'' Programme during which a great part of this work was accomplished.


\section{Recognition of Relative Loop Spaces}\label{sec:recog}

In this section we introduce the notion of $A_\infty$-actions and of $A_\infty$-equivariance. We will prove a recognition theorem according to which any $A_\infty$-action is weakly equivalent to the action of a loop space on a relative loop space in an $A_\infty$-equivariant way. The proof involves the 2-sided bar construction for $A_\infty$-actions.

\subsection{$A_\infty$-spaces and $A_\infty$-maps} \label{sub:ainftymaps}
We begin by recalling the definitions of $A_\infty$-spaces and $A_\infty$-maps. When studying $A_\infty$-maps 
between $A_\infty$-spaces $X$ and $Y$, we will concentrate in the case where $Y$ is a topological monoid with unit. 

We will recall the definition of {\it metric trees}, i.e., trees such that each internal edge has a length in $[0,1]$.  It will be convenient to first introduce the {\it edge-labeled trees} as an assignment of a number to each internal edge. Then the space of metric trees is naturally defined as a quotient of the space of edge-labeled trees.
\subsubsection{Metric trees and edge-labeled trees} 
Let us consider planar oriented finite trees such that some edges are external, i.e., half-open segments attached to a single vertex. Hence our trees are non-compact, albeit finite. We assume that each vertex has at least two incoming edges and that the trees are rooted in the sense that there is only one external edge with an outward orientation. Any external edges that is not the root is inward oriented and will be called a leaf. 

For each tree $T$, its number of leaves will be denoted by $|T|$. The set of internal edges in $T$ will be denoted by $i(T)$ while its cardinality is denoted by $|i(T)|$. 

\begin{defi}
 A tree $T$ endowed with a function $\ell : i(T) \to [0,1]$ will be called an \emph{edge-labeled} tree. The function $\ell$ will be called an \emph{edge-labeling} on $T$.
\end{defi}

For a fixed tree $T$, the space of edge-labellings on $T$ can be identified with $[0,1]^{\times |i(T)|}$. We denote by $\cal T_n$ the space of edge-labeled trees with $n$ leaves, it is topologized as the disjoint union: 
\[ \cal T_n = \bigsqcup_{|T|=n}   [0,1]^{\times |i(T)|} \quad . \] 

Given two edge-labeled trees $T$ and $S$ and $r\in [0,1]$, the edge-labeled tree $T\circ_i^r S$ is obtained by grafting the root
of the tree $S$ on the $i$th leaf  of  $T$  and by assigning the length $r$ to the newly created edge. This operation of 
grafting trees  satisfies the usual operadic relations:
\begin{align}
 U\circ_i^q(T\circ_j^r S)=&(U\circ_i^q T)\circ_{j+i-1}^r S,     & 1\leqslant i\leqslant |U|, 1\leqslant j\leqslant |T|, \label{eq:assoc}\\
 (U\circ_i^q T)\circ_j^r S=&(U\circ_j^r S)\circ_{i+|S|-1}^q T,  & 1\leqslant j<i\leqslant |U|. \label{eq:comm}
\end{align}

In order to define metric trees, we introduce a natural identification $\sim$ on $\cal T_n$ that is defined as follows. If some internal edge is labelled by $0$, then the tree is identified to the tree obtained by collapsing that edge into a single vertex.

\begin{defi}[metric trees]
We define the space of metric trees with $n$-leaves as the quotient $\cal T_n\!\mathrel{/}\sim$, where $\sim$ is the above defined relation. The topology on $\cal T_n\!\mathrel{/}\sim$ is the quotient topology. Any element of $\cal T_n\!\mathrel{/}\sim$ will be called a \emph{metric tree} with $n$-leaves. 
\end{defi}

\subsubsection{$A_\infty$-spaces} 
Let us recall the definition of $A_\infty$-spaces as algebras over the operad $\mathcal K = \{ K_n \}_{n \geqslant 0}$ defined 
by the associahedra. We will use the Boardman and Vogt \cite{BoaVog73} parametrization of $K_n$ by planar metric trees, i.e., we will think of the associahedra $K_n$ as the space $\cal T_n\!\mathrel{/}\sim$ of metric trees with $n$-leaves.

The grafting operation on $\cal T_n$ is well defined on $K_n = \cal T_n\!\mathrel{/}\sim$ and it induces the operad structure of $\cal K$:
\begin{equation*}
\circ_i : K_n \times K_m \to K_{n+m-1}, \quad 1 \leqslant i \leqslant n. 
\end{equation*}
where $T \circ_i S = T \circ^1_i S$. 

The components $K_1$ and $K_0$ are defined as one point spaces. The only point of $K_1$ corresponds to the tree with only one edge and no vertices. It plays the role of the identity in the operad $\cal K$. The only point $\delta_0$ of $K_0$ induces the degeneracy maps $s_i$ in the sense that grafting $\delta_0$ to the $i$th leaf of a tree in $K_n$ is equal to erasing that leaf. 

There is one subtlety in the definition of the degeneracy maps $s_i$ that can be solved as follows.
Suppose that the metric tree $T$ has some vertex with exactly two incoming edges and one of those edges is the $i$th leaf. After erasing that leaf, the adjacent vertex will only have one incoming edge.  For the map $s_i$ to be well defined, we need to include one further relation. The adjacent vertex having only one incoming edge will be erased and the resulting edge will have length equal to the maximum of the lengths of the previous adjacent edges: 
\begin{equation*}
\raisebox{22pt}{
\xymatrix@R=18pt{ 
 \                                                         \\
 *[o]-<5pt>{\bullet}\ar@{-}[u]^(.4){\alpha}\ar@{-}[d]_(.6){\beta}  \\  
 \             }
               }
 \quad \sim \;
 \raisebox{22pt}{
 \xymatrix@R=35pt{ 
 \                        \\  
 \ar@{-}[u]^(.5){{\rm max}\{\alpha,\beta\}}     }
                 } \ .
\end{equation*}
Now we can consider that the degeneracy maps 
\begin{equation*}
 s_i : K_n \to K_{n-1}, \quad 1 \leqslant i \leqslant n, 
\end{equation*}
given by $s_i(\tau) = \tau \circ_i \delta_0$, are well defined. 

There is another operation on edge-labeled trees that will be useful in this paper. Such operations will be called {\it shift} and {\it deshift}. The shift (resp. deshift) operation will be denoted by $T \mapsto \forth{T}$ (resp. $T \mapsto \back{T}$) and is characterized by the following properties for any trees $S$ and $T$ and for any $r \in [0,1]$: 
\begin{equation}\label{def:shift}
\forth{\delta_k} = \delta_k; \quad \   
\forth{S \circ^r_i T} = \forth{S} \circ^r_{i+1} T; \quad \ 
\forth{S \circ^r_{|S|} T} = \forth{T} \circ^r_{1} \forth{S}, 
\hskip 1.5em \mbox{ for } 1 \leqslant i < |S|.  
\end{equation} 
\begin{equation}\label{def:deshift}
\back{\delta_k} = \delta_k; \quad \  
\back{S \circ^r_j T} = \back{S} \circ^r_{j-1} T; \quad \  
\back{S \circ^r_{1} T} = \back{T} \circ^r_{|T|} \back{S},
\hskip 1.5em \mbox{ for } 1 < j \leqslant |S|.  
\end{equation}

One can check that this definition does not depend on the decomposition of a tree as in (\ref{eq:assoc}) and (\ref{eq:comm}).
  The shift $\forth{T}$ can be understood by reversing the orientation of all edges joining the last leaf to the root. On the other hand, 
  the deshift $\back{T}$ can be understood by reversing the orientation of all edges joining the first leaf to the root.  Notice that a shift followed by a deshift is the identity.

\begin{defi}[$A_\infty$-spaces]
An $A_\infty$-space is a topological space $X$ endowed with the structure of an algebra over the operad $\cal K = \{ K_n \}_{n \geqslant 0}$. 
\end{defi}
A $\cal K$-algebra structure on $X$ is 
equivalent to the existence of a family of maps
\begin{equation*}
  M_n : K_n \times X^n \to X, \quad n \geqslant 0, 
\end{equation*}
such that for any $\rho \in K_n$ and $\tau \in K_m$, the following relation holds:
\begin{equation*}
  M_{n+m-1}(\rho \circ_i \tau;x_1,..., x_{n+m-1}) = 
  M_n(\rho; x_1,...,M_m(\tau;x_i,..., x_{i+m-1}),..., x_{n+m-1}). 
\end{equation*}
For $\delta_0 \in K_0$, since the degeneracies are given by $s_i(\rho) = \rho \circ_i \delta_0$, we have in particular:  
\begin{equation*}
  M_{n-1}(s_i(\rho);x_1,\dots , x_{n-1}) = 
  M_n(\rho; x_1,\dots,x_{i-1},e_{{}_X},x_i,\dots, x_{n-1}), 
\end{equation*}
where $1 \leqslant i \leqslant n$ and $e_{{}_X} = M_0(\delta_0)$ is the only point of $X$ in the image of the map $M_0 : K_0 \to X$. 
The element $e_{{}_X}$ will be called the unit of the $A_\infty$-structure on $X$.

In the present paper, unless otherwise stated, all $A_\infty$-spaces are assumed to have a strict unit as above. For $A_\infty$-spaces with homotopy units, we refer the reader to \cite{Muro2011,MuroTonks2011,Iwase2012}. 

\subsubsection{$A_\infty$-maps}
Later in this section we will recall the fact that spaces admitting an $A_\infty$-structure are precisely those spaces having the weak homotopy type of loop spaces. Similarly, maps between $A_\infty$-spaces having the weak homotopy type of loop maps are called $A_\infty$-maps (see \cite{jds:lnm}, Theorem 8.12). 

In the next definition we recall the notion of $A_\infty$-maps $f : X \to Y$ where $X$ is an $A_\infty$-space and $Y$ is a topological monoid (i.e., an associative $H$-space with unit $e_{{}_Y}$). It is possible to define $A_\infty$-maps between two $A_\infty$-spaces in full generality by taking advantage of the $W$-construction introduced by Boardman and Vogt. For any operad $\mathcal P$, let $\mathcal P_{b \to w}$ be the operad whose algebras are pairs of $\mathcal P$-algebras with a morphism between them (for details, see: \cite{MSS02, Markl04}). An $A_\infty$-map between two $A_\infty$-spaces is then defined as the structure of an algebra over the operad $W({\sf Ass}_{b \to w})$ compatible with the given $A_\infty$-structures.  

\begin{defi}[$A_\infty$-maps]
Let $Y$ be a monoid and $X$ an $A_\infty$-space. We say that a map $f:X \to Y$ is an $A_\infty$-map if there exists maps $\{ f_i : K_{i+1} \times X^i \to Y \}_{i \geqslant 1}$ such that $f(x) = f_1(\delta_2;x)$, for all $x \in X$ and: 
\begin{align}
f_k(\rho \circ_{j+1} \tau; x_{1 \dots k}) = & f_{|\rho|-1}(\rho;x_{1 \dots j-1}, M_{|\tau|}(\tau;x_{j \dots j+|\tau|-1}), x_{j+|\tau| \dots k}), \label{eq:amap1}  \\
f_k(\rho \circ_1 \tau; x_{1 \dots k}) = & f_{|\tau|-1}(\tau;x_{1 \dots |\tau|-1}) \cdot f_{|\rho|-1}(\rho;x_{|\tau| \dots k}), \label{eq:amap2}
\end{align}
for $1\leqslant j < |\rho|$, where 
$\rho$ and $\tau$ are metric trees such $|\rho|, |\tau|\geqslant 2$ and $k = |\rho| + |\tau| - 2$.

If moreover the family satisfies: 
\[ f_{n - 1}(s_{j+1}(\rho);x_1, \dots, x_{n-1}) = 
f_n(\rho;x_1, \dots x_{j-1},e_{{}_X},x_j, \dots, x_{n-1}),\ \mbox{ for } n \geqslant 1, \] 
where $f_0(\delta_1) = e_{{}_Y}$, then the $A_\infty$-map is called \emph{unital}.
\end{defi}

Notice that for unital $A_\infty$-maps we have: $f(e_{{}_X}) = e_{{}_Y}$. There is a sequence of polyhedra parameterizing the higher homotopies of an $A_\infty$-map, such polyhedra are known as multiplihedra. The correspondence between the multiplihedra and $W({\sf Ass}_{b \to w})$ is not immediate and involves the so called {\it level trees} introduced in \cite{BoaVog73}. A nice description of such relations, can be found in \cite{forcey,tsutaya}. The multiplihedra can also be obtained as a compactification of a certain class of configuration spaces, as shown by Merkulov in \cite{merkulov11}. 

Given any topological monoid $M$, in the following definition the monoid structure given by the opposite product defined on the space $M$ will be denoted by: $M^{{\rm op}}$.
  
\begin{defi}[$A_\infty$-action]
Given an $A_\infty$-space $X$ and a space $P$, a left $A_\infty$-action of $X$ on $P$ is defined as a unital $A_\infty$-map $X \to {\rm End}(P)$. On the other hand, a right $A_\infty$-action of $X$ on a space $Q$ is defined as a unital $A_\infty$-map $X \to {\rm End}(Q)^{{\rm op}}$.
\end{defi}

For any $A_\infty$-space, if $P$ admits a left (resp. right) $A_\infty$-action, we will say simply that $P$ is a left (resp. right) $X$-space. The structure of left $X$-space on $P$ will be denoted by $X \acinfty P$, while right $X$-space structures will be denoted by $Q \racinfty X$.

\begin{prop} \label{prop:A_infty_action_maps} Let $X$ be an $A_\infty$-space and $P$ a topological space. An $A_\infty$-action structure $X \acinfty P$ is equivalent to a family of maps 
$\{ N_i : K_i \times X^{i-1} \times P \to P \}_{i \geqslant 1}$
such that $N_1 = {\rm Id}_P$ and the following conditions are satisfied:
\begin{enumerate}[i)]
 \item $N_k(\rho \circ_j \tau;x_{1, \dots, k-1},p) = N_{|\rho|}(\rho;x_{1, \dots, j-1},M_{|\tau|}(\tau;x_{j, \dots, j+|\tau|-1}),
 x_{j+|\tau|, \dots ,k-1},p)$, \\ for $1 \leqslant j \leqslant |\rho| - 1$, $|\tau|\geqslant 0$; \\[-10pt] 
 \item $N_k(\rho \circ_{|\rho|} \tau;x_1, \dots, x_{k-1},p) = N_{|\rho|}(\rho;x_1, \dots, x_{|\rho|-1},N_{|\tau|}(\tau;x_{|\rho|}, 
       \dots, x_{k-1},p))$,
 \end{enumerate}
where $\rho$ and $\tau$ are arbitrary metric trees and $k = |\rho| + |\tau| - 1$.
\end{prop}
\begin{proof}
Assuming that a family of maps $N_i$ is given as above, we use the deshift operation (\ref{def:deshift}) on trees to define the unital $A_\infty$-map from $X$ to ${\rm End}(P)$: 
\[ 
f_{n-1}(T;x_{1,...,n-1})(p) = N_n(\back{T};x_{1,...,n-1},p),  \quad \ \mbox{ for } n \geqslant 2. 
\]

To see that $f$ is unital, we take $T = \rho \circ_i \tau$ with $|\tau| = 0$ and use condition {\it i)}. For $n=1$ we observe that $p = N_1(\delta_1;p) = N_2(\delta_2;e_{{}_X},p)$. Hence:
\[
f(e_{{}_X})(p) = f_1(\delta_2;e_{{}_X})(p) = N_2(\delta_2;e_{{}_X},p) = p
\]
in other words: $f(e_{{}_X}) = {\rm Id}_P$. 

Let us check the compatibility with the operadic composition $\circ_1$:
\begin{multline*}
 f_{n-1}(T \circ_1 S ;x_{1,...,n-1})(p) = N_n(\back{T \circ_1 S} ;x_{1,...,n-1},p) = \\ = 
 N_n(\back{S} \circ_{|S|} \back{T} ;x_{1,...,n-1},p) =
 N_{|S|}(\back{S};x_{1,...,|S|-1}, N_{|T|}(\back{T} ;x_{|S|,...,n-1},p)) = \\ =
 f_{|S|-1}(S;x_{1,...,|S|-1}) f_{|T|-1}(T;x_{|S|,...,n-1})(p),
\end{multline*}
where $n = |S| + |T| - 1$. Checking the compatibility for the others $\circ_i{}'s$ is similar. 

We left to the reader the proof of the converse by using the shift $\forth{T}$.
\end{proof}

\begin{obs}
Analogous results hold for right $A_\infty$-actions. For $\tau = \delta_0 \in K_0$, relation {\it i)} is:
$
N_{r-1}(s_j(\rho);x_1, \dots, x_{r-2},p) = N_r(\rho;x_1, \dots, x_{j-1}, e_{{}_X}, x_j, \dots, x_{r-2},p).
$
\end{obs}

\subsection{Moore loops and Moore paths}\label{subsec:moore}

Let $\mathbb{R}^+$ be the space of non-negative real numbers. The space of Moore paths on a topological space $B$ is defined as follows:
\begin{equation*}
   \mpath (B) = \{ (\gamma,r) \in B^{\mathbb{R}^+} \!\!\!\times \mathbb{R}^+ : 
                    \gamma(t) = \gamma(r), \;\; \forall t \geqslant r \}
\end{equation*}
For any $(\gamma,r) \in \mpath (B)$, the number $r$ will be referred to as the \emph{length} of the Moore path $(\gamma,r)$. By abuse of notation we will often denote a Moore path simply by $\gamma$ and its length by $|\gamma|$, such abuse of notation must be understood as: $|(\gamma, r)| = r$.
The operation of cutting path is defined as follows. Given $0\leqslant a<b$ and an element $(\gamma,r)\in   \mpath (B)$ the path $(\gamma_{[a,b]},b-a)$ is defined by
\[\gamma_{[a,b]}(t)=\begin{cases} \gamma(t+a),& 0\leqslant t\leqslant b-a, \\ \gamma(b),& t\geqslant b-a. \end{cases}
\]

The space of based Moore paths on a based topological space $(B,\ast)$ is defined as follows:
\begin{equation*}
   \mpath^\ast (B) = \{ (\gamma,r) \in \mpath (B) : \gamma(0) = \ast \}.
\end{equation*}

The space of Moore loops on $(B,\ast)$ is defined as: 
\begin{equation*}
   \mloop(B) = \{ (\gamma,r) \in \mpath (B) :\gamma(0)=\ast \text{ and } \gamma(r) = \ast \}.
\end{equation*}
All spaces are topologized with the compactly generated topology.


There is a composition on $\mpath(B)$ given by $(\gamma_1, r_1)(\gamma_2,r_2) = (\gamma_1 \cdot \gamma_2, r_1 + r_2)$, for $(\gamma_1, r_1)$ and $(\gamma_2,r_2)$ in $\mpath(B)$. In fact, $\gamma_1(r_1)=\gamma_2(0)$, where $\gamma_1 \cdot \gamma_2$ is the juxtaposition defined by: 
\begin{equation*}
    (\gamma_1 \cdot \gamma_2)(t) = \left\{ \begin{array}{ll}
                                             \gamma_1(t),       & 0 \leqslant t \leqslant r_1, \\  
                                             \gamma_2(t - r_1), & r_1 \leqslant t. 
                                           \end{array} \right. 
\end{equation*}

Such operation defines the structure of a monoid action $\mloop (B) \acts \mpath^{\ast} (B)$. Given an inclusion $A \subseteq B$, the space of relative Moore loops is defined as:
\begin{equation*}
   \mloop(B,A) = \{ (\gamma,r) \in \mpath^\ast(B) : \gamma(r) \in A \}.
\end{equation*} 
Again there is a monoid action $\mloop(B) \acts \mloop(B,A)$ given by juxtaposition.


\subsection{Two-sided Bar Construction for $A_\infty$-actions}
Let us begin by recalling the two-sided bar construction. In our particular case, the definition will be given in terms of the Associahedra, hence there is no need to use the geometrical realization of simplicial objects. Given an $A_\infty$-space $X$ along with right and left $X$-spaces, the 2-sided bar construction amounts to considering planar metric trees whose first leaf is labeled by the elements of the right $X$-space, the last leaf is labeled by elements of the left $X$-space and the remaining leaves are labeled by the elements of the $A_\infty$-space $X$. 

The identifications are essentially those of Boardman-Vogt's $M$-construction whose elements are called {\it cherry trees}. Except that when the first (resp. last) leaf is involved, the identification must take into account the right (resp. left) $A_\infty$-actions on the corresponding spaces. 

\begin{defi}[$2$-sided bar construction]\label{def:2-sided_bar}
Let $(X,\{ M_k \}_{k \geqslant 0})$ be an $A_\infty$-space and let $P$ and $Q$ be left and right $X$-spaces respectively, with structure maps given by: $N_{k} : K_{k} \times X^{\times k-1} \times P \to P$ and $R_{k} : K_{k} \times Q \times X^{\times k-1} \to Q$, for $k \geqslant 1$. 
We define $B(Q,X,P)$ as the following quotient.
\begin{equation*}
     B(Q,X,P)  = \left. \left( \coprod_{n\geqslant 0} K_{n+2} \times Q \times X^{\times n} \times P \right) \right/ \equiv  
\end{equation*}
where $\equiv$ is the relation defined by:
\begin{equation*}
(\rho \circ_{i+1} \tau;x_0, x_1, \dots, x_{n+1}) \equiv (\rho;x_0, \dots,   
      x_{i-1}, E_{|\tau|}(\tau; x_{i, \dots, i+|\tau|-1}),x_{i+|\tau|}, \dots, x_{n+1})
\end{equation*}
where $x_0 \in Q$, $x_{n+1} \in P$ and $|\tau| \geqslant 1$ with $0 \leqslant i \leqslant |\rho| - 1$. In this definition, $E_k$ is just a symbol that must be replaced by $R_k, N_k$ or $M_k$ according to the position (left, right or middle) it occupies in the formula. 
\end{defi}

The topology of $B(Q,X,P)$ is the limit topology of the sequence $B_{n-1}(Q,X,P) \subseteq B_{n}(Q,X,P)$, where $B_n(Q,X,P)$ denotes the image of  
\begin{equation*}
K_n(Q,X,P) = \coprod_{k = 0}^n K_{k+2} \times Q \times X^{\times k} \times P,  
\end{equation*}
under the quotient map and each $B_n(Q,X,P)$ is given the quotient topology.

In terms of metric trees, the identifications in the above definition occur when some internal edge has length $1$. For instance, we have the following identifications:
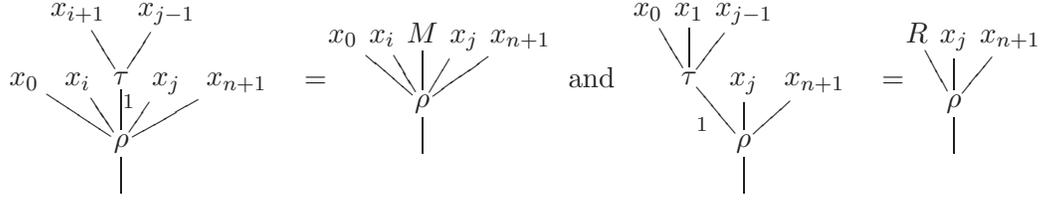
\newcommand{\Aaa}{ *[o]+<4pt>{\tau} \ar@{-}[ul]     \ar@{-}[ur]                                                  }
\newcommand{\Abb}{ *[o]+<4pt>{\rho} \ar@{-}[u]_(.65){\hskip -2.7pt 1} \ar@{-}[ul] \ar@{-}[ur] \ar@{-}[urr] \ar@{-}[ull] \ar@{-}[d] }
\newcommand{\Aab}{ *[o]+<4pt>{\rho} \ar@{-}[u]      \ar@{-}[ul] \ar@{-}[ur] \ar@{-}[urr] \ar@{-}[ull] \ar@{-}[d] }
\newcommand{\Aba}{*[o]+<4pt>{\rho}  \ar@{-}[d] \ar@{-}[u] \ar@{-}[ur] \ar@{-}[ul]}
\[
\raisebox{30pt}{
\xymatrix@R=12pt@C=-2pt{  
        & x_{i+1} &     & x_{j-1}&         \\
    x_0 & x_i     &\Aaa & x_j    & x_{n+1} \\
        &         &\Abb &        &         \\
        &         &     &        &         \\ } \quad \raisebox{-30pt}{=} \hskip -10pt
\raisebox{-10pt}{
\xymatrix@R=12pt@C=-2pt{
    x_0 & x_{i}   &   M  & x_{j}  & x_{n+1}  \\
        &         & \Aab &        &          \\
        &         &      &        &          \\ }}
\raisebox{-30pt}{and} 
\xymatrix@R=12pt@C=-2pt{  
    x_0 &   x_1           & x_{j-1} &         \\
        &\Aaa \ar@{-}[u]  & x_j     & x_{n+1} \\
        &                 & \Aba^1  &         \\
        &                 &         &         \\ } \quad \raisebox{-30pt}{=} \hskip -10pt
\raisebox{-10pt}{
\xymatrix@R=12pt@C=-2pt{
     R & x_j  & x_{n+1}  \\
       &\Aba  &          \\
       &      &          \\ }}
}
\]
where $M = M_{j-i-1}(\tau; x_{i+1,\dots,j-1})$ and $R = R_{j}(\tau;x_0, x_{1,\dots,j-1})$. The above trees are simplified since only part of their leaves are shown. 

Given an $A_\infty$-space $X$, both $X$ and the one point space $\ast$ are $X$-spaces with the canonical and trivial actions respectively.

\begin{defi}
Let $X$ be an $A_\infty$-space with a left $A_\infty$-action $X \acinfty P$.
\begin{enumerate}[i)]
\item The classifying space of $X$ is defined as $B(*,X,*)$ and denoted by $BX$. 
\item The homotopy orbit space of the $A_{\infty}$-action $X \acinfty P$ is defined as $B(\ast,X,P)$ and will be denoted by $P \hquotient X$. 
\end{enumerate}
\end{defi}

\begin{nota}
For each element $(T;q,x_{[n]},p) \in K_{n+2} \times Q \times X^n \times P$, its class in $B(Q,X,P)$ will be denoted by  
$[T;q,x_{[n]},p]$.
\end{nota}

There is an important map $f:X \to \mloop (BX)$ that can be defined as follows. For each $x \in X$, the loop $f(x) \in \mloop (BX)$ of length $2$ is defined by:
\begin{equation*}
  f(x)(t) = \left\{ \begin{array}{ll}
                   {[} \delta_2 \circ_2^{1-t} \delta_2 ; \ast , x, \ast {]},       & 0 \leqslant t \leqslant  1,  \\ 
                   {[} \delta_2 \circ_1^{t-1} \delta_2 ; \ast , x, \ast {]},   & 1 \leqslant t \leqslant  2.
                 \end{array}\right.
\end{equation*}
 
The map $f$ will be called the \emph{usual map}. The following picture illustrates the map $f$. The idea is that $x$ is sitting over a leaf which is sliding along the edges of the tree. It defines a closed loop because the action on $\{ \ast \}$ is trivial. 
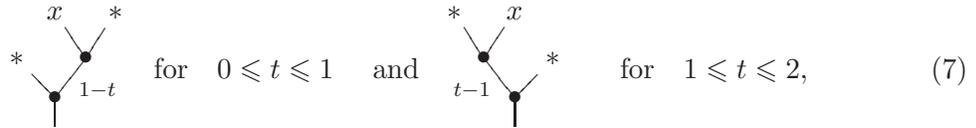
\renewcommand{\Aaa}{*[o]-<5pt>{\bullet}  \ar@{-}[ul] \ar@{-}[ur] }
\renewcommand{\Abb}{*[o]-<5pt>{\bullet}  \ar@{-}[d]  \ar@{-}[ul] \ar@{-}[ur] }
\renewcommand{\Aab}{*[o]-<5pt>{\bullet}  \ar@{-}[d]  \ar@{-}[ur] \ar@{-}[ul] }
\begin{equation} \label{eq:sliding_vertices_simple}
  \raisebox{25pt}{
  \xymatrix@R=6pt@C=2pt{
        &     x        &      & \ast &    \\
  \ast  &              & \Aaa &      &    \\
        & \Abb_{1 - t} &      &      &    \\
        &              &      &      &    \\ }}
  \mbox{for} \quad 0 \leqslant t \leqslant 1 \quad \mbox{ and }
  \raisebox{25pt}{
  \xymatrix@R=6pt@C=2pt{
 \ast &       &        x     &       &      \\
      & \Aaa  &              & \ast  &      \\
      &       & \Aab^{t - 1} &       &      \\
      &       &              &       &      \\ }}
  \quad \mbox{for} \quad 1 \leqslant t \leqslant 2,
\end{equation}

The fact that $P$ is a deformation retract of $B(X,X,P)$ and that the usual map $X \to \mloop (BX)$ is an $A_\infty$-homotopy equivalence are two fundamental facts about the two-sided bar construction. It follows that $B(X,X,P)$ provides a model for $P$ which is appropriate for the study of the $A_\infty$-equivariance of $A_\infty$-actions. Given their importance for our purposes, we will show that those facts are valid for the version of 2-sided bar construction used in the present paper. The idea behind the proofs is the process of sliding edges of trees associated with the homotopy extension property. 

For the next theorems  we need to introduce further notation: an edge-labeled tree $T$  is called comb-reducible if it decomposes as $T=L\circ_1^u R$ for some trees $L,R$ and some $u\in[0,1]$, and comb-irreducible otherwise. Thus any edge-labeled tree $T$ decomposes uniquely as follows:
\begin{equation}\label{T:decomp}
T=T_{k+1}\circ_1^{u_k} T_k\circ_1^{u_{k-1}}\ldots\circ_1^{u_1} T_1
\end{equation}
where the trees $T_i$ are comb-irreducibles. We will call such a decomposition the comb decomposition of a tree.

In what follows, we will define certain Moore paths from a given tree $T$. Since each Moore path must have a length, we first define: 
$l(T)=2(1 + \sum_{r=1}^k u_r)$. Note that: 
\begin{equation*}
l(T\circ_1^r S)=l(T)+l(S)+2r-2.
\end{equation*}
For any tree $T$, the number $l(T)$ will be called the {\it length} of $T$. 
The length is well defined on the quotient space $K_n$. Note that one can extend the length on $K_1$ by $l(\delta_1)=0$.

\begin{thm}\label{thm:usual_map}
    The usual map $X \to \mloop (BX)$ is an $A_\infty$-map.
\end{thm}
\begin{proof}
We need to exhibit maps 
   $f_n : K_{n+1} \times X^n \to \mloop (BX)$ for $n \geqslant 1$,
such that relations (\ref{eq:amap1}) and (\ref{eq:amap2}) are verified and $f_1$ coincides with the usual map $X \to \mloop (BX).$ 
To that end, we use a family of maps $\sigma_{n+1} : K_{n+1} \times \mathbb R^+ \to K_{n+2}$, for $n \geqslant 0$. The existence of a similar family of maps was originally proven by the third author in \cite{Stasheff63} (see Proposition 25). See Lemma \ref{lemma:family_of_maps} for a proof of their existence.

For simplicity, we omit the subscript and for any $T \in K_{n+1}$, $\sigma(T)$ will be a Moore path $\sigma(T) : \mathbb R^+ \to K_{n+2}$  of length $l(T)$, defined above. The properties of $\sigma$ that are relevant for the $A_\infty$-structure are the following:

\begin{subequations} \label{alig:sigma}
  \begin{gather}
      \sigma(T)(0) = \delta_2 \circ_2 \back T \ \mbox{ and } \ \sigma(T)(l(T)) = \delta_2 \circ_1 T;   \label{alig:sigma_0}    \\ 
    \sigma(T_1 \circ_i T_2) = \sigma(T_1) \circ_i T_2,  
     \quad \mbox{ for } \quad 1 < i \leqslant |T_1|;  \label{alig:circ_i} \\ 
    \sigma(T_1 \circ_1 T_2) = (\sigma(T_2) \circ_{|T_2| + 1} \back{T_1} ) \cdot 
                  (\sigma(T_1) \circ_1 T_2),   \label{alig:circ_1}
      \end{gather}
\end{subequations}

where $\back T$ is the deshift (\ref{def:deshift}) and the trees $T_1$ and $T_2$ are such that $|T_1|,|T_2| \geqslant 2$.  In addition we ask that $\sigma(\delta_2)$ coincides with the sliding of edges 
described in (\ref{eq:sliding_vertices_simple}).

Finally, the maps $f_n : K_{n+1} \times X^n \to \mloop (BX)$ are defined by:
\begin{equation*}
    f_n(T;x_{[n]})(t) = [\sigma(T)(t);\ast,x_{[n]},\ast].
\end{equation*}
The fact that such a family verifies the required conditions follows from the properties of $\sigma$ and the definition of the $2$-sided bar construction. 

In fact, the condition on $\sigma(\delta_2)$ implies that $f_1 : X \to \mloop (BX)$ is the usual map. From (\ref{alig:sigma_0}) we have that $f_n(T;x_{[n]})$ is indeed a 
loop because of the trivial action of $X$ on $\{ \ast \}$. 
On the other hand, equation (\ref{alig:circ_i}) along with the identifications in $B(\ast,X,\ast)$ implies that the family of maps 
$\{ f_n \}$ satisfies the condition (\ref{eq:amap1}) in the definition of $A_\infty$-maps.

To finish this proof, we just need to check equation (\ref{eq:amap2}) in the definition of $A_\infty$-maps. We observe that equation (\ref{alig:circ_1}) implies that: \[ f_n(T_1 \circ_1 T_2;x_{[n]}) = f_i(T_2;x_{[i]}) \cdot f_j(T_1;x_{[j]}),\] where $i+j=n$ and the loop $f_i(T_2;x_{[i]})$ given by the movement of the tree $T_2$ is followed by the loop $f_j(T_1;x_{[j]})$ given by the movement of $T_1$, as schematically illustrated (i.e. not showing the true number of leaves) in the following picture:
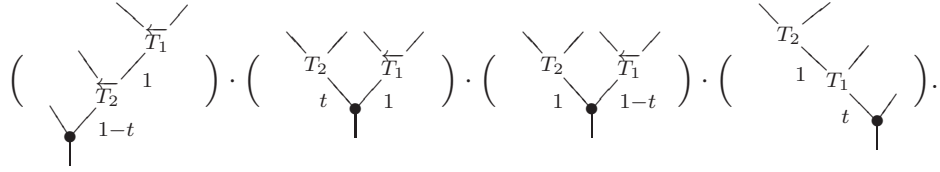
\renewcommand{\Aaa}{*[o]+<5pt>{{}_{\back{T_2}}}  \ar@{-}[ul] \ar@{-}[ur] }
\newcommand{\Baa}{*[o]+<5pt>{{}_{T_2}}  \ar@{-}[ul] \ar@{-}[ur] }
\renewcommand{\Abb}{*[o]-<5pt>{\bullet}                          }
\renewcommand{\Aab}{*[o]-<5pt>{\bullet}  \ar@{-}[d]              }
\renewcommand{\Aba}{*[o]+<5pt>{{}_{\back{T_1}}}  \ar@{-}[ur] \ar@{-}[ul] }
\newcommand{\Bba}{*[o]+<5pt>{{}_{T_1}}  \ar@{-}[ur] \ar@{-}[ul] }
\begin{equation*} 
  \Big(\raisebox{35pt}{
  \xymatrix@R=6pt@C=4pt{
    &                                      &        &      &    \\
    &                                      &        & \Aba &    \\
    &                                      & \Aaa_1 &      &    \\
    & \Aab \ar@{-}[ul] \ar@{-}[ur]_{1 - t} &        &      &    \\
    &                                      &        &      &    \\ }}\Big)\cdot
  \Big(\raisebox{25pt}{
  \xymatrix@R=6pt@C=4pt{
    &      &                                        &       &      \\
    & \Baa &                                        & \Aba  &      \\
    &      & \Aab \ar@{-}[ul]^t \ar@{-}[ur]_1       &       &      \\
    &      &                                        &       &      \\ }}\Big)\cdot
  \Big(\raisebox{25pt}{
  \xymatrix@R=6pt@C=4pt{
    &       &                                         &       &    \\
    & \Baa  &                                         & \Aba  &    \\
    &       & \Aab \ar@{-}[ul]^{1} \ar@{-}[ur]_{1-t}  &       &    \\
    &       &                                         &       &    \\ }}\Big)\cdot
  \Big(\raisebox{35pt}{
  \xymatrix@R=6pt@C=4pt{
    &      &          &                                 &   \\  
    & \Baa &          &                                 &   \\
    &      & \Bba^{1} &                                 &   \\
    &      &          & \Aab \ar@{-}[ul]^t \ar@{-}[ur]  &   \\
    &      &          &                                 &   \\ }}\Big). 
\end{equation*}
where $t \in [0,1]$ for each tree, 
and $\cdot$ denotes the juxtaposition of paths. 
\end{proof}

\begin{lem}\label{lemma:family_of_maps}
There exists a family of maps $\sigma_n : K_n \times \mathbb R^+ \to K_{n+1}$ for $n \geqslant 2$ satisfying the 
properties (\ref{alig:sigma}) presented above.
\end{lem}
\begin{proof}
  We use a recursive argument involving edge-labeled trees. 
  

For any tree $T \in \cal T_n$ the path $ \sigma(T) : \mathbb R^+ \to K_{n+1}$ has the following property:
\begin{equation}\label{eq:sigma_corolla}
   \sigma(T)(t) = \left\{ 
                             \begin{array}{ll}
                               \delta_2 \circ_2^{1-t} \back T, & 0 \leqslant t \leqslant 1, \\
                               \delta_2 \circ_1^{t-l(T)+1} T, & l(T)-1 \leqslant t \leqslant l(T),
                             \end{array}  
                        \right.
\end{equation}
This property defines the path $\sigma(\delta_n)$ for any corolla for $l(\delta_n)=2$.

Let us assume that $\sigma$ is defined for edge-labeled trees with $k<n$ leaves and satisfies (\ref{eq:sigma_corolla}). Now consider a 
tree $T$ such that $T = T_1 \circ_i^r  T_2$

When $1<i $, define: $\sigma(T) = \sigma(T_1) \circ_{i}^r T_2$. 

When $i = 1$, define $\sigma(T)$ as the following juxtaposition of paths
\[ \
(\sigma(T_2)_{[0,l(T_2)+r-1]}\circ_{|T_2|+1}^r \back T_1)\cdot 
(\sigma(T_1)_{[1-r,l(T_1)]}\circ_1^r T_2).
\]
Note that this definition is consistent thanks to (\ref{eq:sigma_corolla}). It is not hard to check that these definitions are independent of the decomposition of
a tree as in (\ref{eq:assoc}) and (\ref{eq:comm}), that $\sigma$ satisfies  (\ref{eq:sigma_corolla}), that $\sigma$ is continuous with respect to $T$ and that it passes to the quotient space $K_n$.
Moreover, properties (\ref{alig:sigma}) follow directly from the construction of $\sigma$. 
\end{proof}

The intuitive description of the paths $\sigma(T)$ is in terms  of {\it sliding vertices along edges in a coherent way}, as illustrated in equation 
(\ref{eq:sliding_vertices}).

\begin{thm}\label{thm:def_retract}
For any $A_\infty$-action $X \acinfty P$, the space $P$ is a deformation retract of $B(X,X,P)$.
\end{thm}
\begin{proof}

The embedding $P \into B(X,X,P)$ is defined by $p \hookrightarrow [\delta_2; e, p]$. 
Now consider the map $B(X,X,P) \to P$ given by $[T;x,x_{[n]},p] \mapsto N_{n+2}(T;x,x_{[n]},p)$, where the $\{N_{i}\}_{i\geqslant 1}'s$ are the structure maps of the $A_\infty$-action $X \acinfty P$. This map is well defined because the equivalence relations in the definition of $B(X,X,P)$ involve the $A_\infty$-structure maps. 
The composition $P \to B(X,X,P) \to P$ is the identity on $P$.

Let us show that $B(X,X,P) \to P \to B(X,X,P)$ is homotopic to the identity. This last map, can be explicitly described as follows: 
\begin{equation*}
  [T;x,x_{[n]},p] \mapsto [\delta_2 \circ_2 T;e,x,x_{[n]},p]. 
\end{equation*}  

The proof is divided in several steps. 
We first build
for any tree $T \in K_n$, a path $\gamma_T : [0,1] \to K_{n+1}$ such that $\gamma_T(0) = T \circ_1 \delta_2$ and $\gamma_T(1) = \delta_2 \circ_2 T$ yielding to a homotopy $\Gamma: K(X,X,P)\times I\rightarrow B(X,X,P)$. Unfortunately this homotopy does not pass to the quotient space $B(X,X,P)$. However,
given two elements $x,y\in K(X,X,P)$ such that $x\equiv y$ we will prove in a second step that $\Gamma(x)$ is homotopic to $\Gamma(y)$. In the last step, we prove, using the homotopy extension property, that we can replace $\Gamma$ by  a $\tilde\Gamma$, homotopic to $\Gamma$ and satisfying  $\tilde\Gamma(x)=\tilde\Gamma(y)$ whenever $x\equiv y$.

\medskip

\noindent {\sl Notation.} Given an edge-labeled tree $T\in \cal T_n$ and its comb decomposition (\ref{T:decomp})
\[
T=T_{r+1}\circ_1^{u_{r}}\ T_r\circ_1^{u_{r-1}}\ldots\circ_1^{u_1} T_1,
\]
we define:

\begin{itemize}
\item $T=R_i\circ_1^{u_i} L_i$, where $R_i=T_{r+1}\circ_1^{u_{r}}\ T_r\circ_1^{u_{r-1}}\ldots\circ_1^{u_{i+1}} T_{i+1}$
and  \\
$L_i=T_{i}\circ_1^{u_{i-1}}\ T_r\circ_1^{u_r}\ldots\circ_1^{u_1} T_1$ for any $1\leqslant i\leqslant r$.
\item $l(T)=2(1+\sum_{k=1}^r u_k)$ 
\item $l_i(T)=1+2\sum_{k=1}^{i-1}u_k+u_i$ for $1\leqslant i\leqslant r$.
\end{itemize}
When there is no confusion, we will write $l,l_i$ instead of $l(T), l_i(T)$.

\medskip

\noindent {\sl First step: the construction of $\Gamma$.} 

The path $\gamma_T:[0,1]\rightarrow K_{n+1}$  is obtained by sliding continuously $\delta_2$ along the path from the left most leaf of $T$ to its root as follows
\begin{equation*}
\gamma_T(t)=\begin{cases}
T\circ_1^{1-lt}\delta_2, &0\leqslant t\leqslant \frac{1}{l}, \\
R_k\circ_1^{u_k}\delta_2\circ_2^{lt-(l_k-u_k)} L_k, &
 \frac{l_k-u_k}{l}\leqslant t\leqslant \frac{l_k}{l},\;  1\leqslant k\leqslant r, \\
 R_k\circ_1^{l_{k}+u_k-lt}\delta_2\circ_2^{u_k} L_k, &
 \frac{l_k}{l}\leqslant t\leqslant \frac{l_{k}+u_k}{l},\; 1\leqslant k\leqslant r, \\
\delta_2\circ_2^{lt-(l_r+u_r)} T, &\frac{l_r+u_r}{l}\leqslant t\leqslant 1. \end{cases}
\end{equation*}

A direct inspection shows that it is well defined, continuous and passes to the quotient space  $K_n$. Furthermore,
the following relations hold, $\forall T,S,u\in [0,1] \text{ and } i\geqslant 2$:
\begin{align}
\gamma_{T\circ_i^u S}=&\gamma_T\circ_{i+1}^u S, \label{def:gammacirci} \\
\gamma_{T\circ_1^u S}(t)=&\begin{cases} T\circ_1^u \gamma_S\left(\frac{l}{l(S)}t\right), &0\leqslant t\leqslant \frac{l(S)+u-1}{l}, \\
\gamma_T\left(\frac{t-\frac{l(S)+2u-2}{l}}{1-\frac{l(S)+2u-2}{l}}\right)\circ_2^u S, &\frac{l(S)+u-1}{l}\leqslant t\leqslant 1. \end{cases} \label{def:gammacirc1}
\end{align}

In addition,  $\forall T\in K_n, \forall t\in[0,1],$ 
\begin{equation}\label{gammawithM}
M_{n+1}(\gamma_T(t);e,x_{[n]})=M_n(T;x_{[n]}).
\end{equation}

The paths $\gamma_T$ define the map
$$\begin{array}{cccc}
\Gamma:& K(X,X,P)\times I&\rightarrow &B(X,X,P)\\
&(T;x;x_{[n]},p)\times t&\mapsto& [\gamma_T(t);e,x,x_{[n]},p]\end{array}$$

\medskip

\noindent {\sl Second step: first homotopies.}

\smallskip

\noindent From equation (\ref{def:gammacirci}), $ \forall t\in [0,1] \text{ and } i\geqslant 1$:
$$\Gamma(T\circ_{i+1} S;x,x_{[n]},p)(t)=\Gamma(T;x,x_{1,\ldots,i-1},E_{|S|}(S;x_{i,\ldots,i+|S|-1}),x_{i+|S|,\ldots,n},p)(t).$$
It implies that $\Gamma$ passes to the quotient by the equivalence relation of definition \ref{def:2-sided_bar}  when $i\geqslant 1$. It remains to check the case $i=0$.
From equation (\ref{def:gammacirc1}) and relation (\ref{gammawithM})
\begin{equation*}
\Gamma(T\circ_1 S,x,x_{[n]},p)(t)=\begin{cases} [T;M_{|S|}(S;x,x_{1,\ldots, |S|-1}),x_{|S|,\ldots,n},p] ,&0\leqslant t\leqslant \frac{l(S)}{l}, \\
\Gamma(T,M_{|S|}(S;x,x_{1,\ldots, |S|-1}),x_{|S|,\ldots,n},p)\left(\frac{t-\frac{l(S)}{l}}{1-\frac{l(S)}{l}}\right), &\frac{l(S)}{l}\leqslant t\leqslant 1. \end{cases}
\end{equation*}

Consequently there is a homotopy between $\Gamma(T\circ_1 S)$ and $\Gamma(T)$:
\begin{multline*}
H_{(T\circ_1 S,T)}(x,x_{[n]},p)(t,u)=\\
\begin{cases} [T;M_{|S|}(S;x,x_{1,\ldots, |S|-1}),x_{|S|,\ldots,n},p] ,&0\leqslant t\leqslant (1-u) \frac{l(S)}{l}, \\
\Gamma(T,M_{|S|}(S;x,x_{1,\ldots, |S|-1}),x_{|S|,\ldots,n},p)\left(\frac{t-(1-u)\frac{l(S)}{l}}{1-(1-u)\frac{l(S)}{l}}\right), &(1-u) \frac{l(S)}{l}\leqslant t\leqslant 1. \end{cases}
\end{multline*}

Before continuing the proof, we need some technical facts whose proofs are left to the reader.
For $\alpha_1,\alpha_2,\alpha_3: I\rightarrow Y$ such that $\alpha_1(1)=\alpha_2(0)$ and $\alpha_2(1)=\alpha_3(0)$ and for
$0<a<b<1$, define:
\[\!\!\!
\begin{array}{ccc}
(\alpha_1*^a\alpha_2)(u)=\begin{cases} \alpha_1(\frac{u}{a}), &\!\!\!0\leqslant u\leqslant a, \\
\alpha_2(\frac{u-a}{1-a}), &\!\!\!a\leqslant u\leqslant 1, \end{cases}&\!\!\!\text{ and }\!\!\!&
(\alpha_1*^a\alpha_2*^b\alpha_3)(u)=\begin{cases} \alpha_1(\frac{u}{a}), &\!\!\!0\leqslant u\leqslant a, \\
\alpha_2(\frac{u-a}{b-a}), &\!\!\!a\leqslant u\leqslant b, \\
\alpha_3(\frac{u-b}{1-b}), &\!\!\!b\leqslant u\leqslant 1. \end{cases}
\end{array}
\]
The following relation holds:
\begin{equation*}
(\alpha_1*^{\frac{a}{b}}\alpha_2)*^b\alpha_3=\alpha_1*^a\alpha_2*^b\alpha_3=\alpha_1*^a(\alpha_2*^{\frac{b-a}{1-a}}\alpha_3).
\end{equation*}
For $0<a<b<c$, let $x_{c;a,b}:=\frac{a(c-b)}{b(c-a)}\in ]0,1[$. For any $0<a_1<a_2<a_3<c$, 
\begin{multline}\label{F:fond}
(\alpha_1*^{x_{c;a_1,a_2}}\alpha_2)*^{x_{c;a_2,a_3}}\alpha_3=\\
\alpha_1*^{x_{c;a_1,a_3}}\alpha_2*^{x_{c;a_2,a_3}}\alpha_3=\alpha_1*^{x_{c;a_1,a_3}}(\alpha_2*^{x_{c-a_1;a_2-a_1,a_3-a_1}}\alpha_3)
\end{multline}
Consequently:
\begin{multline}\label{F:H}
H_{(T,T_3)}(x,x_{[n]},p)(t)=\\
H_{(T,T_3\circ_1 T_2)}(x,x_{[n]},p)(t)*^{x_{l;l_1,l_1+l_2}}H_{(T_3\circ_1 T_2,T_3)}(M_{|T_1|}(x,x_{1,\ldots,|T_1|-1}),x_{|T_1|,\ldots,n},p)(t),
\end{multline}
where $T=T_3\circ_1T_2\circ_1 T_1$, $l_1=l(T_1)$, $l_2=l(T_2)$, $l=l(T)$ and 
where the composition of paths is relative to the variable $u$ in the definition of $H$.

Note that for any trees $T,S,V$ and $i\geqslant 1$ one has
\begin{multline}\label{def:Hcirci}
H_{(T\circ_1 S,T)}(x,x_{1,\ldots,i-1},E_{|V|}(V;x_{i,\ldots,i+|S|-1}),x_{i+|S|,\ldots,n},p)=\\
\begin{cases}
H_{((T\circ_1 S)\circ_{i+1} V,T)}(x,x_{[n]},p), &1\leqslant i+1\leqslant |S|,\\
H_{((T\circ_1 S)\circ_{i+1} V,T\circ_{i+1-|S|} V)}(x,x_{[n]},p), &|S|+1\leqslant i+1\leqslant |T|+|S|-1.\end{cases}
\end{multline}

\noindent{\sl Last step: construction of $\tilde\Gamma$.}

In this step we will use induction to build homotopies
$\Psi_n: K_{n}(X,X,P)\times I\times I\rightarrow B(X,X,P)$ satisfying:

\begin{align}\label{condi}
\begin{split}
\Psi_{n+2}(T;x,x_{[n]},p)(t,0)&=\Gamma(T;x,x_{[n]},p)(t), \\
\Psi_{n+2}(T;x,x_{[n]},p)(0,u)&=[T;x,x_{[n]},p],   \\
\Psi_{n+2}(T;x,x_{[n]},p)(1,u)&= [\delta_2;e,N_{n+2}(T;x,x_{[n]},p)],  
\end{split}
\end{align}

\begin{align}
\Psi_{n+2}(T\circ_{i+1} S;x,x_{[n]},p)&=
 \Psi_{|T|}(T;x,x_{1,\ldots,i-1},E_{|S|}(S;x_{i,\ldots,i+|S|-1}),x_{i+|S|,\ldots,n},p),\label{condcirci} 
 \end{align}
for $i>0$ and
\begin{align}
\Psi_{n+2}(T\circ_1 S,x,x_{[n]},p)(t,1)&= \Psi_{|T|}(T;M_{|S|}(x,x_{1,\ldots,|S|-1}),x_{|S|,\ldots,n},p)(t,1). \label{condcirc1}
\end{align}
As a consequence the path $\tilde\Gamma:K(X,X,P)\times I\rightarrow B(X,X,P)$ defined by 
$$\tilde\Gamma(T;x,x_{[n]},p)(t)=\Psi_{|T|}(T;x,x_{[n]},p)(t,1)$$ is well defined on $B(X,X,P)$ and provides the required deformation retract.

We start by $\Psi_2(\delta_2;x,p)(t,u)=\Gamma(\delta_2;x,p)(t).$ Let us assume that $\Psi_{k}$ is defined for $k\leqslant n+1$ and satisfies the above 
conditions.

\medskip

Since the inclusion of the boundary $\partial K_n$ into $K_n$ is a cofibration, then the map
$$\partial K_{n+2}\times X^{n+1}\times P\times I\cup K_{n+2}\times X^{n+1}\times P\times \partial I\rightarrow K_n(X,X,P)\times I$$
is also a cofibration. Then using the homotopy extension property, it is enough to define $\Psi$ on
$$Y_n=(\partial K_{n+2}\times X^{n+1}\times P\times I\cup K_{n+2}\times X^{n+1}\times P\times \partial I)\times I\cup K_n(X,X,P)\times I\times\{0\}. $$

Equations (\ref{condi}), and (\ref{condcirci}) implies that  $\Psi$ is defined by induction on $Y_n$ 
except for the elements of the form $(T\circ_1 S,x,x_{[n]},p)$. We then define
\begin{multline*}
\Psi_{n+2}(T\circ_1 S,x,x_{[n]},p)(t)=\\
H_{(T\circ_1 S,T)}(x,x_{[n]},p)(t)*^{x_{l;l(S),l-1}}\Psi_{|T|}(T;M_{|S|}(S;x,x_{1,\ldots,|S|-1}),x_{|S|,\ldots,n},p)(t).
\end{multline*}
Since $l=l(T)+l(S)$ and $l(T)\geqslant 2$, then $l(S)<l-1<l$ and the composition is well defined because
\begin{multline*}
H_{(T\circ_1 S,T)}(x,x_{[n]},p)(t,1)=\Psi_{|T|}(T;M_{|S|}(S;x,x_{1,\ldots,|S|-1}),x_{|S|,\ldots,n},p)(t,0)=\\
\Gamma(T;M_{|S|}(S;x,x_{1,\ldots,|S|-1}),x_{|S|,\ldots,n},p)(t).
\end{multline*}
Moreover, relations (\ref{F:fond}) and (\ref{F:H}) imply that this definition is independent of the decomposition of a tree  
$T=T_1\circ_1 S_1=T_2\circ_1 S_2$. Similarly relations (\ref{condcirci}) and  (\ref{def:Hcirci})  imply that this definition is independent of the decomposition of a tree  
$T=T_1\circ_1 S_1=U\circ_{i+1} V$, with $i>0$.
Finally, one can check that this definition satisfies (\ref{condi}),  and  (\ref{condcirc1}). \end{proof}

The following corollary is obtained by projecting onto $B(\ast,X,P)$ the inverse of $\tilde\Gamma$.

\begin{coro}\label{cor:alpha} For any $A_\infty$-action $X \acinfty P$, there exists a map 
$$\alpha_P: B(X,X,P)\rightarrow \mpath(B(\ast,X,P))$$
such that
\begin{align*}
\alpha_P([T;x;x_{[n]},p])(0)=&[\delta_2;\ast,N_{n+2}(T;x;x_{[n]},p)]\\
\alpha_P([T;x;x_{[n]},p])(1)=&[T;\ast;x_{[n]},p]\\
\alpha_P([\delta_2,x,p])(t)=&\begin{cases} [\delta_2\circ_2^{1-2t}\delta_2;\ast,x,p],& 0\leqslant t\leqslant \frac{1}{2}, \\
[\delta_2\circ_1^{2t-1}\delta_2;\ast,x,p],& \frac{1}{2}\leqslant t\leqslant 1. \end{cases}
\end{align*}
\end{coro}

In the next theorem we state, without proof, another well known and important property of the $2$-sided bar construction. Before stating the theorem, let us recall that an $H$-space $X$ is called \emph{grouplike} when its product induces a group structure on $\pi_0(X)$. When the $H$-space in question is an $A_\infty$-space, then the grouplike property implies that all possible translations \[ M_n(T;x_{1, \dots, i-1}, \underline{\ \ }, x_{i+1, \dots, n}) : X \to X \]  
are weak equivalences, for all $n \geqslant 2, T \in K_n$ and $x_j \in X$.

\begin{thm}
If $X$ is grouplike, then $B(X,X,P) \to B(\ast,X,P)$ is a quasi-fibration.
\end{thm}

\begin{thm} If $X$ is  grouplike, then the usual map $f:X \to \mloop (BX)$ is a weak equivalence.
\end{thm}

\begin{proof}
Corollary \ref{cor:alpha} applied to $P=\ast$ with the trivial $A_\infty$-action implies that $\alpha_\ast: B(X,X,\ast)\to\mpath^{\ast}(BX)$ is a weak equivalence since
 $B(X,X,\ast)$ is contractible as well as $\mpath^{\ast}(BX)$. Moreover $\alpha_\ast$ restricted to $X$ has value in $\mloop(BX)$ and is homotopic to the usual map. Finally $B(X,X,\ast) \to B(\ast,X,\ast)$ is a quasi-fibration and the usual five lemma argument shows that the usual map $X \to \mloop (BX)$ is a weak equivalence.
\end{proof}

\subsection{$A_\infty$-equivariance}\label{sub:equivariance}
Let us now describe the notion of $A_\infty$-equivariance for maps between spaces admitting $A_\infty$-actions. Our approach will be analogous to the one given for $A_\infty$-maps in section \ref{sub:ainftymaps}. 

\begin{defi}\label{def:ainftyequivariance}
Consider an $A_\infty$-space $(X,\{ M_n \}_{n \geqslant 0})$ and an $X$-space $P$ with structure maps $\{ N_j : K_j \times X^{j-1} \times P \to P \}_{j \geqslant 1}$ and a monoid $Y$ with a monoid action $Y \acts Q$. We say that a map $F : P \to Q$ is $A_\infty$-equivariant with respect to an $A_\infty$-map $\{ f_i : K_{i+1} \times X^i \to Y \}_{i \geqslant 1}$ if there is a family $\{ F_n : K_{n+1} \times X^{n-1} \times P \to Q \}_{n \geqslant 1}$ such that $F_1 = F$ and, for any $\tau \in K_{i+1}$ and $\rho \in K_{j+1}$, satisfies the following conditions:
\begin{equation*}
F_n(\tau \circ_k \rho; x_{[n-1]},p) = \left\{
    \begin{array}{ll}
      \!\!\! F_i(\tau; x_{1,\ldots,k-2},M_{j+1}(\rho;x_{k-1, \dots, k+j-1}),x_{k+j,\ldots,n-1},p), 
      &  2 \leqslant k \leqslant i; \\ 
      \!\!\! F_i(\tau; x_{1,\dots,i-1},N_{j+1}(\rho;x_{i,\dots,n-1},p)), & k = i+1; \\
      \!\!\! f_i(\rho; x_{1,\dots,j}) \cdot F_j(\tau;x_{j+1,\dots,n-1},p), & k = 1. 
    \end{array}\right.  
\end{equation*}
where $i+j = n$ and the action $M \acts Q$ is denoted by: $m \cdot q$.
\end{defi}

Any $A_\infty$-action $X \acinfty P$ can be canonically extended to an action on the (unreduced) cone $CP$. Since the action $X \acinfty P$ in general does not fix the base point $p_0 \in P$, we need to work with the unreduced cone. 
The natural embedding $P \into \mloop(CP \hquotient X,P \hquotient X)$ is defined by
\[ p\mapsto \gamma_p(t)=\left[ \delta_2; \ast, [p,t]\right],\]
where we recall that $P\hquotient X:=B(\ast,X,P)$.
We will show that the natural embedding  is $A_\infty$-equivariant with respect to  $X \to \mloop B(\ast,X,CP)$, the $A_\infty$-map obtained as the composition:
\begin{equation}\label{eq:natural_map}
   X \to \mloop B(\ast,X,\ast) \hookrightarrow \mloop B(\ast,X,CP).
\end{equation}
where $X \to \mloop BX$ is the usual map of Theorem \ref{thm:usual_map}.


\begin{prop}\label{prop:equivariance}
 The natural embedding $P \into \mloop(CP \hquotient X,P \hquotient X)$ sending $p$ to $\gamma_p$ is $A_\infty$-equivariant with respect to the $A_\infty$-map $X \to \mloop(CP \hquotient X)$ defined by 
 \emph{(\ref{eq:natural_map})}. 
\end{prop}
\begin{proof}
The proof will follow the lines of the proof of theorem \ref{thm:usual_map}

We will exhibit a family of maps 
\[ F_n : K_{n+1} \times X^{n-1} \times P \to {\mpath}^{\ast}(CP \hquotient X), \quad {\rm for } \quad n \geqslant 1. \] 
To any $(T;x_{[n-1]},p) \in K_{n+1} \times X^{n-1} \times P$, the map $F_n$ associates the following juxtaposition of paths:
\begin{equation*}
 \Lambda(T;x_{[n-1]},p) \cdot  \big[T;*,x_{[n-1]},[p,t] \big]
\end{equation*}
where $\Lambda$ defines the path $\Lambda(T;x_{[n-1]},p)(t) = \big[ \lambda(T,t); \ast, x_{[n-1]},[p,0] \big]$, where $\lambda$ is defined as a family of maps 
$\lambda_n : K_n \to \mpath(K_n)$ whose construction is well tailored to ensure 
$\{ F_n \}_{n \geqslant 1}$ has the properties of definition \ref{def:ainftyequivariance}. More precisely, $\lambda(T)$ satisfies:
\begin{align} \label{alig:lambda} 
\begin{split}
   & \lambda(T)(0) = \delta_2 \circ_2 s_{|T|}(\back T) \ \mbox{ and } \ \lambda(T)(|\lambda(T)|) = T;     \\ 
   & \lambda(T_1 \circ_i T_2) = \lambda(T_1) \circ_i T_2,  
     \quad \mbox{ for } \quad 1<i \leqslant |T_1|;   \\ 
   & \lambda(T_1 \circ_1 T_2) = (\sigma(T_2) \circ_{|T_2| + 1} s_{|T_1|}(\back T_1) ) \cdot 
                  (\lambda(T_1) \circ_1 T_2), 
\end{split}
\end{align} 
These conditions ensure that $\{ F_n \}_{n \geqslant 1}$ has the properties of definition \ref{def:ainftyequivariance}.

As in the proof of lemma \ref{lemma:family_of_maps}, we will construct the family $\lambda_n : K_n \to \mpath(K_n)$ recursively. However, the case in which the vertex adjacent to the root has valence (i.e., number of incoming edges) equal to $2$ will be treated separately. 

{\sl First case:} We define $\lambda(T)$ for trees $T$ whose valence of the vertex adjacent to the root is $>2$. 
In this case,
the Moore path $ \lambda(T)$ has length $l(T)-1$ and  is such that:
\begin{equation}\label{eq:lambda_corolla}
   \lambda(T,t) =  \delta_2 \circ_2^{1-t} s_{|T|}\back T,  \quad 0 \leqslant t \leqslant 1.
   \end{equation}
In particular, this defines $\lambda(\delta_n)$ for any corolla $\delta_n$ with $n\geqslant 3$, since $|\lambda(\delta_n)|=l(\delta_n)-1=1$.

Let us assume that $\lambda$ is defined for edge-labeled trees with $k<n$ leaves and whose number of incoming edges at the root is $>2$,
and that it satisfies (\ref{eq:lambda_corolla}). 
Let us now consider a tree $T$ such that $T = T_1 \circ_i^r  T_2$.

For $1<i $, we define: $\lambda(T) = \lambda(T_1) \circ_{i}^r T_2$. On the other hand, for $i = 1$ the path
$\lambda(T)$ is defined as the following juxtaposition of paths: 
\[ 
\big(\sigma(T_2)_{[0,l(T_2)+r-1]}\circ_{|T_2|+1}^r s_{|T_1|}(\back T_1)\big) \cdot 
\big(\lambda(T_1)_{[1-r,l(T_1)-1]}\circ_1^r T_2\big). 
\]
where we have used the operation $\gamma_{[a,b]}$ defined in subsection \ref{subsec:moore}.

It is not hard to check that these definitions are independent of the decomposition of
a tree as in (\ref{eq:assoc}) and (\ref{eq:comm}), that $\lambda$ satisfies  (\ref{eq:lambda_corolla}), 
that $\lambda$ is continuous with respect to $T$ and that it passes to the quotient space $K_n$.
The properties (\ref{alig:lambda})
follow directly from the construction of $\lambda$.

{\sl Second case:}  We define $\lambda(T)$ where the number of incoming edges at the root of $T$ is $2$. Let us first define $\lambda(\delta_2)=\delta_2$ as the constant path with length $0$. 

Now, let us consider $T=(\delta_2\circ_2^v T_2)\circ_1^u T_1$ and define $\lambda(T)$ as the following juxtaposition:
\[
\big(\sigma(T_1)_{[0,l(T_1)+u-1]}\circ_{|T_1|+1}^{Max(u,v)} T_2\big) \cdot 
\big((\delta_2\circ_2^{Max(u,v)-t} T_2)_{[0,Max(u,v)-v]}\circ_1^u T_1\big)
\]
Notice that in particular we have: 
$\lambda(\delta_2 \circ_1^u T_1) = \lambda((\delta_2\circ_2^1 \delta_1)\circ_1^u T_1)=\sigma(T_1)_{[0,l(T_1)+u-1]}$ and
$\lambda(\delta_2 \circ_2^v T_2)=\lambda((\delta_2\circ_2^v T_2)\circ_1^1 \delta_1) =( \delta_2 \circ_2^{1-t} T_1)_{[0,1-v]}$.
It is necessary to check that this construction is well defined on the quotient $K_n$, in particular one has to check that it is compatible with the first case when $u$ or $v$ tends to $0$. The properties (\ref{alig:lambda}) must also be checked. The verification of these facts are long but straightforward.
\end{proof}

For the $(n+1)$-corolla $\delta_{n+1} \in K_{n+1}$, it follows that $F_n(\delta_{n+1};x_{[n-1]},p)$ is defined by first ``sliding vertices along edges'' as follows: 
\renewcommand{\Aaa}{*[o]-<5pt>{\bullet}  \ar@{-}[ul] \ar@{-}[u]  \ar@{-}[ur] }
\renewcommand{\Abb}{*[o]-<5pt>{\bullet}  \ar@{-}[d] \ar@{-}[ul] \ar@{-}[ur]  }
\renewcommand{\Aab}{*[o]-<5pt>{\bullet}  \ar@{-}[d]  \ar@{-}[ur] \ar@{-}[ul] }
\begin{equation} \label{eq:sliding_vertices}
  \raisebox{25pt}{
  \xymatrix@R=12pt@C=8pt{
  {}\save[]  +<25pt,6pt>*{x_{[n-1]}} +<2pt,-8pt>*{\overbrace{\hskip 2ex}}  +<28pt,6pt>*{[p,0]} \restore
        &              &      &      &    \\
  \ast  &              & \Aaa &      &    \\
        & \Abb_{1 - s} &      &      &    \\
        &              &      &      &    \\ }}
  \mbox{for} \quad 0 \leqslant s \leqslant 1 \quad \mbox{ and }
  \raisebox{25pt}{
  \xymatrix@R=12pt@C=8pt{
  {}\save[]  +<22pt,6pt>*{x_{[n-1]}} +<2pt,-8pt>*{\overbrace{\hskip 2ex}}  +<25pt,-15pt>*{[p,0]} \restore
 \ast &       &              &       &      \\
      & \Aaa  &              &       &      \\
      &       & \Aab^{s - 1} &       &      \\
      &       &              &       &      \\ }}
  \quad \mbox{for} \quad 1 \leqslant s \leqslant 2,
\end{equation}
and ending with the path $[\delta_2;\ast;[p,t]]$, taking the base point to $[\delta_2;\ast;[p,1]]$.

\begin{thm}[Recognition Theorem of Relative Loop Spaces]\label{thm:mainresult} For any $A_\infty$-action $X \acinfty P$, there is a weak equivalence 
$P \to \mloop(CP \hquotient X, P \hquotient X)$ that is $A_\infty$-equivariant with respect to the natural map $X \to \mloop(CP \hquotient X)$. 
\end{thm}

\begin{proof}The $A_\infty$-equivariance is the statement of Proposition \ref{prop:equivariance}.  Let us show that the map $P \to \mloop (CP \hquotient X, P \hquotient X)$ is a weak equivalence.  Let us consider
$\beta: B(X,X,P)\to \mloop (CP \hquotient X, P \hquotient X)$  obtained as the composition of paths
$$\beta([T;x,x_{[n]},p])=\left[\delta_2;\ast,[N_{n+2}(x,x_{[n]},p),t]\right]\cdot \iota\alpha_P([T;x,x_{[n]},p])$$
where $\iota$ is the embedding $\mpath(P\hquotient X)\into \mpath(CP\hquotient X)$ and $\alpha_P$ has been defined in Corollary \ref{cor:alpha}.
Now $\beta$ induces a map of quasi-fibrations:



\begin{equation*}
\raisebox{45pt}{
\xymatrix{
X            \ar[r] \ar@{^{(}->}[d]  &  {\rm ev}^{-1}([p_0,1])                  \ar@{^{(}->}[d]   \\ 
B(X,X,P)     \ar[r]^<<<<{\beta} \ar[d]           &  \mloop(CP \hquotient X, P \hquotient X) \ar[d]^{{\rm ev}} \\ 
P\hquotient X  \ar@{=}[r]         &  P \hquotient X,                                
}}
\end{equation*}
where ${\rm ev}:  \mloop(CP \hquotient X, P \hquotient X) \to P \hquotient X$ is the evaluation map that takes each relative loop to its final point in $P \hquotient X$.

Since $CP$ is contractible, projecting onto its tip gives a weak equivalence $ {\rm ev}^{-1}([p_0,1]) \to \mloop(BX)$. The top map followed by
this weak equivalence is homotopic to the usual map $X\into \mloop(BX)$ which is a weak equivalence. Hence the top map is a weak equivalence and 
 the usual five lemma applies to show that the middle map is a weak equivalence. As a consequence the weak equivalence $\beta$ composed with
 the usual inclusion $P \into B(X,X,P)$ which is known to be a homotopy equivalence (see Theorem \ref{thm:def_retract}) is a weak equivalence. The map so obtained is given by 
 $$p\mapsto (\underbrace{\left[\delta_2;\ast,[p,t]\right]}_{\gamma_p(t)})\cdot \underbrace{( \left[\delta_2\circ_2^{1-2t}\delta_2;\ast,e_X,[p,1]\right])\cdot (\left[\delta_2\circ_1^{2t}\delta_2;\ast,e_X,[p,1]\right])}_{\rho_p}.$$
 It is then enough to prove that $\rho_p$ is homotopic to a constant path in order to prove that the map $p\mapsto \gamma_p$ is a weak equivalence.
 The path $\rho_p$ lifts to a map $\tilde\rho_p:I\to B(X,X,P)$ replacing $\ast$ by $e_X$. The map $H_p:I\times I\to B(\ast,X,P)$ defined by $H_p(t,u)=\alpha_P(\tilde\rho_p(t),u)$ is a homotopy from the constant path at $[\delta_2,\ast,p]$ to $\rho_p$ thanks to Corollary \ref{cor:alpha}.
 \end{proof}





\section{The 2-colored operad of $A_\infty$-actions}\label{shaction}

In this section, we study the approach to $A_\infty$-actions by means of Boardman-Vogt $W$-construction $W(\ac)$ on the operad of monoid actions $\ac$. We also compare our definition with the one given by Iwase and Mimura in \cite{iwase-mimura}. 

\subsection{The operad \ac}
The 2-colored non-symmetric operad of monoid actions on spaces will be denoted by $\ac$, i.e., it is the operad whose algebras consist of a topological monoid $X$ and a left $X$-space $P$. The colors in this case will be called closed and open and denoted $\{ \cl, \op \}$, according to the following convention: 
\begin{itemize}
  \item To any topological monoid $X$, there will be assigned the color {\it closed} $\cl$.
  \item The color {\it open} $\op$ will be assigned to any $X$-space $P$.
\end{itemize}
Such choice of color comes from Theorem \ref{thm:mainresult} where the elements of the monoid correspond to loops while the elements of the $X$-space correspond to relative loops to which we assign the color {\it open}, because relative loops are not necessarily closed. 
\begin{defi}[Operad of left actions]\label{def:act}
Let us define $\ac$ as the $\{ \cl, \op \}$-colored operad such that $\ac(p,0;\cl)$ 
and $\ac(p,1;\op)$ are singletons for any $p \geqslant 0$, and $\ac(p,q;{\rm x})$ is the empty space otherwise.
\end{defi}
An algebra over $\ac$ consists of a pair $(X,P)$ such that $X$ is a topological monoid and $P$ is a left $X$-space. The unit $e$ is the base point of $X$.
The action of $X$ on $P$ is given by the only point of $\ac(1,1;\op)$. The compatibility between action and the product follows from the fact that $\ac(n,1;\op)$ is a singleton.
The product on $X$ is given by the only point of $\ac(2,0;\cl)$, its associativity follows from the fact that each $\ac(n,0;\cl)$ is a single point space and the identity element exists because of the degeneracy induced by $\ac(0,0;\cl)$. For the same reason, the action determined by $\ac$ is unital in the sense that $e_{{}_X} \cdot p = p, \quad \forall p \in P$. 

The operad of non-unital actions will be denoted by $\overline{\ac}$. It coincides with $\ac$ except when $p=q=0$, in which case it is defined as the empty space: $\overline{\ac}(0,0;\cl) = \emptyset$.


\subsection{The operad $\ac_\infty$} \label{sec:w(act)}
Let us now introduce the operad $\ac_\infty$ whose algebras are precisely $A_\infty$-actions $X \acinfty P$. We will give a definition of $\ac_\infty$ in terms of the Boardman and Vogt's $W$-construction. In the next section we will explore in details the link between $A_\infty$-actions and the $1$-dimensional Swiss-cheese operad $\SC_1$.

Since in this work the $A_\infty$-actions are strictly unital, in the following definition we use the $W$-construction on the operad of non-unital actions $\overline{\ac}$ and augment it by including the strict unit, i.e., by defining $\ac_\infty(0,0;\cl)$ as the one point space $K_0$. 
\begin{defi}
The operad $\ac_\infty$ is defined as follows: 
$\ac_\infty(p,q;{\rm x}) = W\!\!\left(\overline{\ac}\right)(p,q;{\rm x})$, for $p+q>0$, and $\ac_\infty(0,0;\cl) = K_0$.
\end{defi}



Since $\ac(p,0;{\rm x})$ and $\ac(p,1;\op)$ are spaces with a single point, the trees in $\ac_\infty$ can be viewed as planar metric trees with colored edges. The edges of color $\cl$ (closed) will be represented by straight line segments. When the color is open, this will be indicated by the symbol $\op$ over the edge. For instance, we have the trees
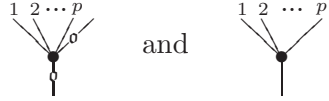
\[
\raisebox{15pt}{
 $\xymatrix@R=16pt{ {}\save[]   +<-15pt,-1pt>*{\text{\tiny $1$}}
                                +<8pt,0pt>*{\text{\tiny $2$}}
                                +<8pt,0pt>*{...}
                                +<8pt,0pt>*{\text{\tiny $p$}}
                                 \restore \\ 
                    *[o]-<5pt>{\bullet}\ar@{-}[u]!<-20pt,0pt>
                                       \ar@{-}[u]!<-10pt,0pt>
                                       \ar@{-}[u]!< 10pt,0pt>
                                       \ar@{-}[u]!< 20pt,0pt>|-{\op}
                                       \ar@{-}[d]!<0pt,5pt>|-{\op}        \\  
                                       \                          }$
                } \quad \makebox{ and } \quad
\raisebox{15pt}{
 $\xymatrix@R=16pt{ {}\save[]   +<-15pt,-1pt>*{\text{\tiny $1$}}
                                +<8pt,0pt>*{\text{\tiny $2$}}
                                +<11pt,0pt>*{...}
                                +<11pt,0pt>*{\text{\tiny $p$}}
                                 \restore \\ 
                    *[o]-<5pt>{\bullet}\ar@{-}[u]!<-20pt,0pt>
                                       \ar@{-}[u]!<-10pt,0pt>
                                       \ar@{-}[u]!< 20pt,0pt>
                                       \ar@{-}[d]!<0pt,5pt>                  \\  
                                       \                          }$
                }
\]                       
belonging respectively to $\ac_\infty(p,1;\op)$ and $\ac_\infty(p,0;\cl)$. 

Since the open leaf of a tree in $\ac_\infty(p,1;\op)$ is the last leaf of a planar tree, and since  
$\ac_\infty(0,0;\cl) = K_0$ will induce the degeneracy for closed leaves, it follows from Proposition \ref{prop:A_infty_action_maps} that $\ac_\infty$-algebras are precisely $A_\infty$-actions. 

Iwase and Mimura \cite{iwase-mimura} have defined $A_\infty$-actions along a given map as follows.
  
\begin{defi}[$A_\infty$-action after Iwase-Mimura]\label{def:iwase-mimura}
 Given an $A_\infty$-space $(X, \{M_s\})$, a space $P$ and a map $h : X \to P$, an (left) $A_\infty$-action of $X$ on $P$ along $h$ is a sequence of maps $\{ N_i : K_i \times X^{i-1} \times P \to P \}_{i \geqslant 1}$, such that $N_1 = Id_P$ and:
 \begin{enumerate}[i)]
 \item $N_k(\rho \circ_	j \tau;x_1, \dots, x_{k-1},p) = N_r(\rho;x_1, \dots, x_{j-1},    M_s(\tau;x_j, \dots, x_{j+s-1}), \dots, x_{k-1},p)$, for $1 \leqslant j \leqslant r - 1.$ \\[-10pt] 
 \item $N_k(\rho \circ_r \sigma;x_1, \dots, x_{k-1},p) = N_r(\rho;x_1, \dots, x_{k - s};N_s(\sigma;x_{k-s+1}, \dots, x_{k-1},p))$,\\[-10pt]
 \item $N_k(\mu;x_1, \dots, x_{k-1},p_0) = h(M_{k-1}(s_k(\mu);x_1, \dots, x_{k-1}))$.
 \end{enumerate}
  where $k = r + s - 1, \ \rho \in K_r$, $\tau, \sigma \in K_s$, $\mu \in K_k$ and $p_0$ is the base point of $P$.
 \end{defi}

\begin{obs}
Notice that when $s=0$, condition $i)$ says that 
 \[ N_k(\mu;x_1,..., x_{i-1},e, x_{i+1}..., x_{k-1},p) = N_{k-1}(s_i(\mu);x_1,..., \widehat{x_i},..., x_{k-1},p), \] for any $1 \leqslant i \leqslant k-1$. If we impose the condition $s \geqslant 1$, then we get a ``non-unital $A_\infty$-action'' (i.e., an $A_\infty$-action under which the unit $e$ does not necessarily act trivially).
\end{obs}

From Proposition \ref{prop:A_infty_action_maps}, it follows that any family of maps as in the above definition gives rise to an $A_\infty$-action. The converse is true up to $A_\infty$-equivariant weak equivalence because of Theorem \ref{thm:mainresult}. Indeed, let $(B,A)$ be a pair of topological spaces and consider the action $\Omega_{{}_M}(A) \acts \Omega_{{}_M}(B,A)$. By choosing a path $p_0 \in \Omega_{{}_M}(B,A)$ as the base point, the map $h : \Omega_{{}_M}(B) \to \Omega_{{}_M}(B,A)$ can be defined as $h(\gamma) = \gamma \cdot p_0$. Since it is an action in the strict sense, it follows immediately that it is an action along $h$ as defined above. 

Let us close this section by observing that an $W\!\!\left(\overline{\ac}\right)$-algebra structure on a pair $(X,P)$ is equivalent to the existence of a family of maps as in definition \ref{def:iwase-mimura} satisfying only relations $i)$ and $ii)$ with the condition $s \geqslant 1$. 
Such structures will be called {\it non-unital $A_\infty$-actions}. 


\section{The one-dimensional Swiss-cheese operad} \label{sec:swiss_cheese}

In this section we will explore the relations between the operad of $A_\infty$-actions and the one dimensional Swiss-cheese operad. 
Such comparison will imply in particular  the existence of an $A_\infty$-action $\Omega(B) \acinfty \Omega(B,A)$ in the space of usual loops and usual paths endowed with the Poincar\'e product of loops and juxtaposition of paths.

\subsection{The Operad $\mathcal{SC}_1$} 

The one dimensional Swiss-cheese operad is a 2-colored operad. The elements of color {\emph open} consist of those configurations of little intervals inside the radius $1$ interval centered at the origin: $[-1,1]$ that are invariant under the map $x \mapsto -x$. On the other hand, the elements color {\emph closed} are the usual little intervals in $[0,1]$.  

\subsubsection{Composition law}
The space of usual little intervals configurations will be denoted by $\mathcal{SC}_1(n,0;\cl)$ while the space of open little intervals configurations will be denoted by $\mathcal{SC}_1(n,i;\op)$, where $i=0$ when there is no little interval in $[-1,1]$ that is centered at the origin and $i=1$ for those open symmetric configurations of little intervals in $[-1,1]$ having a little interval centered at the origin. 
 
The composition law $\mathcal{SC}_1(n,0;\cl)$ is the usual composition law of the little intervals operad. The composition law for $\mathcal{SC}_1(n,i;\op)$ goes as follows. First the composition   
\begin{equation}
\circ^{\op} : \mathcal{SC}_1(p,1;\op) \times \mathcal{SC}_1(q,i;\op) \to \mathcal{SC}_1(p + q,i;\op)
\end{equation}
is given by the usual gluing operation in the little interval centered at the origin, and
\begin{equation}
\circ_i^{\cl} : \mathcal{SC}_1(k,i;\op) \times \mathcal{SC}_1(n,0;\cl) \to \mathcal{SC}_1(k + n - 1,i;\op)
\end{equation}
is given by first symmetrizing the configuration in $[0,1]$ into a symmetric configuration in $[-1,1]$ and then gluing the resulting positive and negative little intervals of the new symmetric configuration accordingly into the corresponding little intervals of the $i$th pair of conjugated little intervals in $\mathcal{SC}_1(k,1)$, as illustrated by Figure \ref{1dim_insertion} where the gluing is from the bottom to the top.

\begin{figure}[htb]
\setlength{\unitlength}{3700sp}%
\begingroup\makeatletter\ifx\SetFigFont\undefined%
\gdef\SetFigFont#1#2#3#4#5{%
  \reset@font\fontsize{#1}{#2pt}%
  \fontfamily{#3}\fontseries{#4}\fontshape{#5}%
  \selectfont}%
\fi\endgroup%
\begin{picture}(3906,1680)(2326,-2773)
\thinlines
{\color[rgb]{0,0,0}\put(3451,-1411){\line(-1, 0){150}}
\put(3301,-1411){\line( 0,-1){300}}
\put(3301,-1711){\line( 1, 0){150}}
}%
{\color[rgb]{0,0,0}\put(3901,-1411){\line( 1, 0){150}}
\put(4051,-1411){\line( 0,-1){300}}
\put(4051,-1711){\line(-1, 0){150}}
}%
{\color[rgb]{0,0,0}\put(2701,-1411){\line(-1, 0){150}}
\put(2551,-1411){\line( 0,-1){300}}
\put(2551,-1711){\line( 1, 0){150}}
}%
{\color[rgb]{0,0,0}\put(2851,-1411){\line( 1, 0){150}}
\put(3001,-1411){\line( 0,-1){300}}
\put(3001,-1711){\line(-1, 0){150}}
}%
{\color[rgb]{0,0,0}\put(4951,-1411){\line( 1, 0){150}}
\put(5101,-1411){\line( 0,-1){300}}
\put(5101,-1711){\line(-1, 0){150}}
}%
{\color[rgb]{0,0,0}\put(4501,-1411){\line(-1, 0){150}}
\put(4351,-1411){\line( 0,-1){300}}
\put(4351,-1711){\line( 1, 0){150}}
}%
{\color[rgb]{0,0,0}\put(5701,-1411){\line( 1, 0){150}}
\put(5851,-1411){\line( 0,-1){300}}
\put(5851,-1711){\line(-1, 0){150}}
}%
{\color[rgb]{0,0,0}\put(5551,-1411){\line(-1, 0){150}}
\put(5401,-1411){\line( 0,-1){300}}
\put(5401,-1711){\line( 1, 0){150}}
}%
{\color[rgb]{0,0,0}\put(3076,-2386){\line(-1, 0){ 75}}
\put(3001,-2386){\line( 0,-1){150}}
\put(3001,-2536){\line( 1, 0){ 75}}
}%
{\color[rgb]{0,0,0}\put(3301,-2386){\line( 1, 0){ 75}}
\put(3376,-2386){\line( 0,-1){150}}
\put(3376,-2536){\line(-1, 0){ 75}}
}%
{\color[rgb]{0,0,0}\put(2701,-2386){\line(-1, 0){ 75}}
\put(2626,-2386){\line( 0,-1){150}}
\put(2626,-2536){\line( 1, 0){ 75}}
}%
{\color[rgb]{0,0,0}\put(2776,-2386){\line( 1, 0){ 75}}
\put(2851,-2386){\line( 0,-1){150}}
\put(2851,-2536){\line(-1, 0){ 75}}
}%
{\color[rgb]{0,0,0}\put(5326,-2386){\line( 1, 0){ 75}}
\put(5401,-2386){\line( 0,-1){150}}
\put(5401,-2536){\line(-1, 0){ 75}}
}%
{\color[rgb]{0,0,0}\put(5101,-2386){\line(-1, 0){ 75}}
\put(5026,-2386){\line( 0,-1){150}}
\put(5026,-2536){\line( 1, 0){ 75}}
}%
{\color[rgb]{0,0,0}\put(5701,-2386){\line( 1, 0){ 75}}
\put(5776,-2386){\line( 0,-1){150}}
\put(5776,-2536){\line(-1, 0){ 75}}
}%
{\color[rgb]{0,0,0}\put(5626,-2386){\line(-1, 0){ 75}}
\put(5551,-2386){\line( 0,-1){150}}
\put(5551,-2536){\line( 1, 0){ 75}}
}%
{\color[rgb]{0,0,0}\put(3751,-2311){\line(-1, 0){150}}
\put(3601,-2311){\line( 0,-1){300}}
\put(3601,-2611){\line( 1, 0){150}}
}%
{\color[rgb]{0,0,0}\put(4651,-2311){\line( 1, 0){150}}
\put(4801,-2311){\line( 0,-1){300}}
\put(4801,-2611){\line(-1, 0){150}}
}%
{\color[rgb]{0,0,0}\multiput(3301,-2311)(100.00000,0.00000){2}{\line( 1, 0){ 50.000}}
\multiput(3451,-2311)(0.00000,-120.00000){3}{\line( 0,-1){ 60.000}}
\multiput(3451,-2611)(-100.00000,0.00000){2}{\line(-1, 0){ 50.000}}
}%
{\color[rgb]{0,0,0}\multiput(2701,-2311)(-100.00000,0.00000){2}{\line(-1, 0){ 50.000}}
\multiput(2551,-2311)(0.00000,-120.00000){3}{\line( 0,-1){ 60.000}}
\multiput(2551,-2611)(100.00000,0.00000){2}{\line( 1, 0){ 50.000}}
}%
{\color[rgb]{0,0,0}\multiput(5101,-2311)(-100.00000,0.00000){2}{\line(-1, 0){ 50.000}}
\multiput(4951,-2311)(0.00000,-120.00000){3}{\line( 0,-1){ 60.000}}
\multiput(4951,-2611)(100.00000,0.00000){2}{\line( 1, 0){ 50.000}}
}%
{\color[rgb]{0,0,0}\multiput(5701,-2311)(100.00000,0.00000){2}{\line( 1, 0){ 50.000}}
\multiput(5851,-2311)(0.00000,-120.00000){3}{\line( 0,-1){ 60.000}}
\multiput(5851,-2611)(-100.00000,0.00000){2}{\line(-1, 0){ 50.000}}
}%
{\color[rgb]{0,0,0}\put(2401,-2461){\line( 1, 0){3600}}
}%
{\color[rgb]{0,0,0}\put(2551,-2161){\line(-1, 0){150}}
\put(2401,-2161){\line( 0,-1){600}}
\put(2401,-2761){\line( 1, 0){150}}
}%
{\color[rgb]{0,0,0}\put(5851,-2161){\line( 1, 0){150}}
\put(6001,-2161){\line( 0,-1){600}}
\put(6001,-2761){\line(-1, 0){150}}
}%
{\color[rgb]{0,0,0}\put(4201,-2311){\line( 0,-1){300}}
}%
{\color[rgb]{0,0,0}\put(2401,-1561){\line( 1, 0){3600}}
}%
{\color[rgb]{0,0,0}\put(2551,-1261){\line(-1, 0){150}}
\put(2401,-1261){\line( 0,-1){600}}
\put(2401,-1861){\line( 1, 0){150}}
}%
{\color[rgb]{0,0,0}\put(5851,-1261){\line( 1, 0){150}}
\put(6001,-1261){\line( 0,-1){600}}
\put(6001,-1861){\line(-1, 0){150}}
}%
{\color[rgb]{0,0,0}\put(4201,-1411){\line( 0,-1){300}}
}%
{\color[rgb]{0,0,0}\multiput(4201,-1561)(-15.30612,-18.36735){50}{\makebox(1.6667,11.6667){\SetFigFont{5}{6}{\rmdefault}{\mddefault}{\updefault}.}}
}%
{\color[rgb]{0,0,0}\multiput(4201,-1561)(14.87805,-18.59756){49}{\makebox(1.6667,11.6667){\SetFigFont{5}{6}{\rmdefault}{\mddefault}{\updefault}.}}
}%
{\color[rgb]{0,0,0}\multiput(6001,-1561)(-3.92603,-23.55619){39}{\makebox(1.6667,11.6667){\SetFigFont{5}{6}{\rmdefault}{\mddefault}{\updefault}.}}
}%
{\color[rgb]{0,0,0}\multiput(2401,-1561)(3.91536,-23.49218){39}{\makebox(1.6667,11.6667){\SetFigFont{5}{6}{\rmdefault}{\mddefault}{\updefault}.}}
}%
\put(5956,-1201){\makebox(0,0)[lb]{\smash{{\SetFigFont{12}{14.4}{\rmdefault}{\mddefault}{\updefault}$1$}}}}
\put(4126,-1216){\makebox(0,0)[lb]{\smash{{\SetFigFont{12}{14.4}{\rmdefault}{\mddefault}{\updefault}$0$}}}}
\put(2326,-1201){\makebox(0,0)[lb]{\smash{{\SetFigFont{12}{14.4}{\rmdefault}{\mddefault}{\updefault}$-1$}}}}
\end{picture}%
\caption{Illustration of $\circ_1^{\cl} : \SC_1(1,1;\op) \times \SC_1(2,0;\cl) \to \SC_1(2,1;\op)$.}  
\label{1dim_insertion}
\end{figure}
Note also that one can interpret $\SC_1(n,1,\op)$ as  
the configuration of $n+1$ little intervals such that the last little interval  contains $1$. We will use this interpretation in the sequel.
\subsubsection{Example}
 
That the classical $[0,1]$-parameterized based loop space admits an $A_\infty$-structure
is a straightforward generalization of the standard proof of homotopy associativity, which involves only a homotopy defined on $[0,1]$. The specific pentagonal higher homotopy illustrated in \cite{jds:lnm} involves precisely  \emph{dyadic} homeomorphisms \cite{dehornoy}
of $[0,1]$ to itself as do the higher homotopies.

The $[0,1]$-parameterized space of relative loops on a pair $A \subseteq B$ is defined as: 
\begin{equation*} 
  \Omega(B,A) = \{ \gamma : [0,1] \to B : \gamma(0) = \ast \ \makebox{ and } \gamma(1) \in A \}. 
\end{equation*}

To see that the pair $(\Omega B,\Omega(B,A))$ is an $\mathcal{SC}_1$-algebra, we observe that the closed part of $\mathcal{SC}$ is isomorphic to the little intervals operad and it acts on $\Omega B$ in the usual way. For the open part of $\mathcal{SC}$, 
%
%
the action on $(\Omega B,\Omega(B,A))$ is given by:
\begin{equation*}
  \rho : \mathcal{SC}_1(n,1,\op) \times \Omega(B)^{\times n} \times \Omega(B,A) \to \Omega(B,A)
\end{equation*}
where the path $\rho(d;\ell_1, \dots, \ell_n,\gamma)$ is obtained by running each loop $\ell_i$ through the cor\-responding little interval,
 as one usually does in loop spaces and then the path $\gamma$ through the subinterval of $d$ containing $1$. The remaining points of $[0,1]$ are taken to the base point of $B$. From the definition of $\mathcal{SC}_1(n,1,\op)$, the configuration $d$ will always have a little interval containing $1$, hence if $\gamma$ is a path that is not a loop, then $\rho(d;\ell_1, \dots, \ell_n,\gamma)$ is also a path that is not a loop, i.e., the map $\rho$ respects the color.

\subsection{Homotopy equivalence}

There exists an operad morphism $\theta : W\!\!\left(\overline{\ac}\right) \to \mathcal{SC}_1$ so that $\theta^{\cl}(n,0) : W\!\!\left(\overline{\ac}\right)(n,0;\cl) \to \mathcal{SC}_1(n,0;\cl)$ and $\theta^{\op}(n,1) : W\!\!\left(\overline{\ac}\right)(n,1;\op) \to \mathcal{SC}_1(n,1;\op)$ are homotopy equivalences for any $n \geqslant 1$.


It is possible to exhibit $\theta$ in an elementary way as follows.  
Each corolla in $W\!\!\left(\overline{\ac}\right)(n,0;\cl)$ is taken to a configuration of little intervals in $\mathcal{SC}_1(n,0;\cl)$ where each subinterval has diameter $\frac{1}{n}$, while each corolla in $\mathcal{SC}_1(n,1;\op)$ is taken to a  of little intervals in $\mathcal{SC}_1(n,1;\op)$ 
where each subinterval has diameter $\frac{1}{n+1}$, thus the configuration contains 1. The compatibility with the operad structure is obtained 
via a convex combination. To be more precise we define $\theta$ by double induction, on the number of leaves of a tre $T$ and then on the number of internal leaves. If the latter number is zero, $T$ is a corolla and $\theta$ has been defined on corollas. Then, one defines for any tree 
$T=X\circ_i^r Y$ in $W\!\!\left(\overline{\ac}\right)$ the configuration $\theta(T)$ as $(1-r)\theta(X)\circ_i \theta(Y)+r\theta(X\circ_i^0 Y).$


From the existence of the above operad homotopy equivalence, it follows the existence of a non-trivial 
$W\!\!\left(\overline{\ac}\right)$-algebra structure on the pair: $(\Omega(B,A),\Omega B)$. In other words, there is a {\it non-unital $A_\infty$-action} of $\Omega B$ on $\Omega(B,A)$.

We observe that the strict unit plays a crucial role in Theorem \ref{thm:mainresult}. Hence one cannot use the above operad homotopy equivalence $\theta$ to prove the {\it Swiss-cheese recognition principle}: ``every $SC_1$-algebra is homotopy equivalent to a pair of the form $(\Omega(B,A),\Omega B)$''. 

The Swiss-cheese recognition principle is indeed true in general for the Swiss-cheese operad $\SC_n$ with $n \geqslant 1$, as will be shown in a sequel to this paper. On the other hand, without Theorem \ref{thm:mainresult} one cannot show that any $A_\infty$-action $X \acinfty P$ (defined just as an unital $A_\infty$-map $X \to {\rm End}(P))$ is homotopy equivalent to a $SC_1$-algebra. Hence, the results of the present paper are not particular cases of the Swiss-cheese recognition principle.





\begin{thebibliography}{MSS02}

\bibitem[BM07]{BerMoe07}
Clemens Berger and Ieke Moerdijk.
\newblock Resolution of coloured operads and rectification of homotopy
  algebras.
\newblock {\em Contemp. Math.}, 431:31--58, 2007.

\bibitem[BV73]{BoaVog73}
J.~Michael Boardman and Rainer~M. Vogt.
\newblock {\em Homotopy invariant algebraic structures on topological spaces}.
\newblock Springer-Verlag, Berlin, 1973.
\newblock Lecture Notes in Mathematics, Vol. 347.

\bibitem[Deh11]{dehornoy}
Patrick Dehornoy.
\newblock {Tamari lattices and the symmetric Thompson monoid,}
\newblock {\em Progress in Math.}, vol. 299, Birkhauser (2012) pp. 211-250


\bibitem[For08]{forcey}
Stefan Forcey.
\newblock Convex hull realizations of the multiplihedra.
\newblock {\em Topology Appl.}, 156(2):326--347, 2008.

\bibitem[Hil53]{hilton:book}
P.~J. Hilton.
\newblock {\em An introduction to homotopy theory}.
\newblock Cambridge University Press, Cambridge, UK, 1953.

\bibitem[Hoe12]{hoefel12}
Eduardo Hoefel.
\newblock Some elementary operadic homotopy equivalences.
\newblock {Guillermo Corti\~nas (ed.) Topics in Noncommutative Geometry.
  Proceedings of the Conference, Buenos Aires, Argentina, Clay Mathematics
  Proceedings, vol.16, 67-74}, 2012.

\bibitem[IM89]{iwase-mimura}
Norio Iwase and Mamoru Mimura.
\newblock {Higher homotopy associativity.}
\newblock {Algebraic topology, Proc. Int. Conf., Arcata/Calif. 1986, Lect.
  Notes Math. 1370, 193-220 (1989).}, 1989.

\bibitem[Iwa12]{Iwase2012}
Norio Iwase.
\newblock Associahedra, multiplihedra and units in {$A_{\infty}$} form.
\newblock arXiv:1211.5741, 2012.

\bibitem[Mar04]{Markl04}
Martin Markl.
\newblock Homotopy algebras are homotopy algebras.
\newblock {\em Forum Math.}, 16(1):129--160, 2004.

\bibitem[Mer11]{merkulov11}
Sergei~A. Merkulov.
\newblock Operads, configuration spaces and quantization.
\newblock {\em Bull. Braz. Math. Soc. (N.S.)}, 42(4):683--781, 2011.

\bibitem[MSS02]{MSS02}
Martin Markl, Steve Shnider, and Jim Stasheff.
\newblock {\em Operads in algebra, topology and physics}, volume~96 of {\em
  Mathematical Surveys and Monographs}.
\newblock American Mathematical Society, Providence, RI, 2002.

\bibitem[MT11]{MuroTonks2011}
Fernando Muro and Andrew Tonks.
\newblock Unital associahedra.
\newblock arXiv:1110.1959, 2011.

\bibitem[Mur11]{Muro2011}
Fernando Muro.
\newblock Homotopy units in {A}-infinity algebras.
\newblock arXiv:1111.2723, 2011.

\bibitem[Now71]{nowlan}
Robert~A. Nowlan.
\newblock {$A\sb n$-actions on fibre spaces.}
\newblock {\em Math. J., Indiana Univ.}, 21:285--313, 1971.

\bibitem[Poi10]{poincare}
Henri Poincar{\'e}.
\newblock {\em Papers on topology}, volume~37 of {\em History of Mathematics}.
\newblock American Mathematical Society, Providence, RI, 2010.
\newblock {{\i}t Analysis situs} and its five supplements, Translated and with
  an introduction by John Stillwell.

\bibitem[Sta63]{Stasheff63}
James Stasheff.
\newblock Homotopy associativity of {$H$}-spaces. {I}, {II}.
\newblock {\em Trans. Amer. Math. Soc. 108 (1963), 275-292; ibid.},
  108:293--312, 1963.

\bibitem[Sta70]{jds:lnm}
J.~Stasheff.
\newblock {\em H-Spaces from a homotopy point of view}.
\newblock Springer-Verlag, Berlin, 1970.
\newblock Lecture Notes in Mathematics, Vol. 161.

\bibitem[Sug57]{sugawara:h}
M~Sugawara.
\newblock On a condition that a space is an {H}-space.
\newblock {\em Math. J. Okayama Univ.}, 6:109--129, 1957.

\bibitem[Tsu12]{tsutaya}
Mitsunobu Tsutaya.
\newblock Finiteness of {$A_n$}-equivalence types of gauge groups.
\newblock {\em J. Lond. Math. Soc. (2)}, 85(1):142--164, 2012.

\bibitem[Vog03]{Vogt2003}
R.~M. Vogt.
\newblock Cofibrant operads and universal {$E_\infty$} operads.
\newblock {\em Topology Appl.}, 133(1):69--87, 2003.

\end{thebibliography}
%


\def\cprime{$'$} \def\cprime{$'$} \def\cprime{$'$} \def\cprime{$'$}
  \def\cprime{$'$} \def\cprime{$'$} \def\cprime{$'$} \def\cprime{$'$}

\end{document}